\documentclass[12pt,fleqn]{article}

\usepackage{amsmath}
\usepackage{amsthm}
\usepackage{amssymb}
\usepackage{amsfonts}
\usepackage{epsf}
\usepackage{graphicx}
\usepackage{epstopdf}
\usepackage{hyperref}
\usepackage{color}

\newtheorem{lem}{Lemma}

\newtheorem{thm}{Theorem}

\newtheorem{prop}{Proposition}

\theoremstyle{remark}
\newtheorem{rmk}{Remark}

\theoremstyle{definition}
\newtheorem{defn}{Definition}

\DeclareMathOperator*\Res{{Res}} 
 \DeclareMathOperator\re{{Re}}
 \numberwithin{equation}{section}

\newcommand{\D}{\displaystyle}

\numberwithin{equation}{section}

\newcounter{comment}
\def\l{\lambda}

\begin{document}
\title{Hankel determinants for a singular complex weight  and the first and third Painlev\'{e}  transcendents}

\author{Shuai-Xia Xu$^a$, Dan Dai$^b$ and Yu-Qiu Zhao$^c$\footnote{Corresponding author (Yu-Qiu Zhao).
 {\it{E-mail
address:}} {stszyq@mail.sysu.edu.cn} }  }
  \date{
  \small
 {\it{$^a$Institut Franco-Chinois de l'Energie Nucl\'{e}aire, Sun Yat-sen University, GuangZhou
510275,  China}}\\
{\it{$^b$Department of Mathematics, City University of Hong Kong, Tat Chee Avenue, Kowloon, Hong Kong}}\\
 {\it{$^c$Department of Mathematics, Sun Yat-sen University, GuangZhou
510275, China}}
}

\maketitle

\begin{abstract}
  In this paper, we consider polynomials orthogonal with respect to a varying   perturbed Laguerre weight  $e^{-n(z-\log z+t/z)}$ for $t<0$ and $z$ on certain contours in the complex plane.
   When the parameters $n$, $t$ and the degree $k$ are fixed, the   Hankel determinant for the singular complex weight   is shown to be the
    isomonodromy  $\tau$-function  of the Painlev\'e III equation.
    When the degree $k=n$, $n$ is  large and  $t$ is close to a critical value, inspired by the study of the Wigner time delay in quantum transport, we show that  the double scaling  asymptotic behaviors of the recurrence coefficients and the Hankel determinant are described in terms of a Boutroux tronqu\'ee solution to the Painlev\'e I equation. Our approach is based on the Deift-Zhou nonlinear steepest descent method for Riemann-Hilbert problems.

\end{abstract}





\newpage


\section{Introduction and statement of results}

Let $w_t(x)$ be the following singularly perturbed Laguerre weight
\begin{equation}\label{pL weight}
w_t(x)=w(x;t)=e^{-nV_t(x)}, \qquad x \in (0,+\infty)
\end{equation}
with
\begin{equation} \label{potential}
  V_t(x)=x-\log x +\frac{t}{x},~~t\geq 0.
\end{equation}
The Hankel determinant is defined as
\begin{equation}\label{Hankel}
D_k[w(x;t)]=\det(\mu_{i+j})_{i,j=0}^{k-1},
\end{equation}
where $\mu_j$ is the $j$-th moment of $w_t(x)$, namely,
\begin{equation*}
\mu_j=\int_{0}^{\infty} x^{j}w_t(x)dx.
\end{equation*}
Note that when $t\geq 0$, the integral in the above formula is convergent so that the Hankel determinant $D_k[w;t]=D_k[w(x;t)]$ in \eqref{Hankel} is well-defined. Moreover, it is well-known that the Hankel determinant can be expressed as
\begin{equation}\label{D-n and leading coeff}
  D_k[w;t] = \prod_{j=0}^{k-1} \gamma_{j,n}^{-2}(t);
\end{equation}see \cite[p.28]{Szego},
where $\gamma_{k,n}(t)$ is the leading coefficient of the $k$-th order polynomial orthonormal with respect to the weight function in \eqref{pL weight}. Or, let $\pi_{k,n}(x)$ be the $k$-th order monic orthogonal polynomial, then $\gamma_{k,n}(t)$ appears in the following orthogonal relation
\begin{equation*}
\int_{0}^{\infty}\pi_{k,n}(x)x^j e^{-nV_t(x)}dx=\gamma_{k,n}^{-2}(t)\delta_{jk}, \qquad j=0,1,\cdots,k
\end{equation*}
for fixed $n$. Moreover,  the monic  orthogonal polynomials $\pi_{k,n}(x)$ satisfy a three-term recurrence relation as follows:
\begin{equation} \label{recurrence relation}
  x\pi_{k,n}(x) = \pi_{k+1,n} (x) + \alpha_{k,n}(t) \pi_{k,n}(x) + \beta_{k,n}(t) \pi_{k-1,n}(x),~~k=0,1,\cdots,
\end{equation}with $\pi_{-1, n}(x)\equiv 0$ and $\pi_{0, n}(x)\equiv 1$,
where the appearance of  $n$ and $t$ in the coefficients indicates  their  dependence  on $n$ and  the parameter $t$ in the varying weight \eqref{pL weight}.

In this paper, however, we will focus on the case when $t<0$. Since all the  above  integrals on $[0, \infty)$  become divergent for negative  $t$, we need to deform  the integration path from the positive real axis to certain curves in the complex plane. Consequently, the orthogonality will be converted to the \emph{non-Hermitian orthogonality} in the complex plane. More precisely, let us define the following new weight function
on $\Gamma=\Gamma_1\cup\Gamma_2\cup\Gamma_3$:
\begin{equation}\label{pL new weight}
w_t(z)=w(z;t)=c_j e^{-nV_t(z)},~~z \in \Gamma_j,~~\textrm{with}~c_1=1,c_2=\alpha,c_3=1-\alpha,
\end{equation}where  $\alpha$ is a complex constant, the curves $\Gamma_1=(2\delta,\infty)$,
                     {  $\Gamma_2=\{\delta  (1+e^{i\theta}  )|\; \theta\in(0,\pi)\}$ and $\Gamma_3=\{\delta   (1+e^{i\theta}  )| \; \theta\in(-\pi,0)\}$;}
see Figure \ref{contour-ortho}, $\delta$ being a positive constant. The potential is  defined in the cut plane $\mathbb{C}\setminus (-\infty, 0]$ as
\begin{equation} \label{potential new}
  V_t(z)=z-\log z +\frac{t}{z},~~\arg z\in (-\pi, \pi),~~t< 0.
\end{equation} The {  orthogonality}  relation now takes the form
\begin{equation}\label{pL polynomials: orthogonality:complex}
 \int_{\Gamma}\pi_{k,n}(z)z^{j}w_t(z)dz=\gamma_{k,n}^{-2}(t)\delta_{jk},~~j=0,1, \cdots,  k.
\end{equation}
\begin{figure}[h]
 \begin{center}
   \includegraphics[width=9cm]{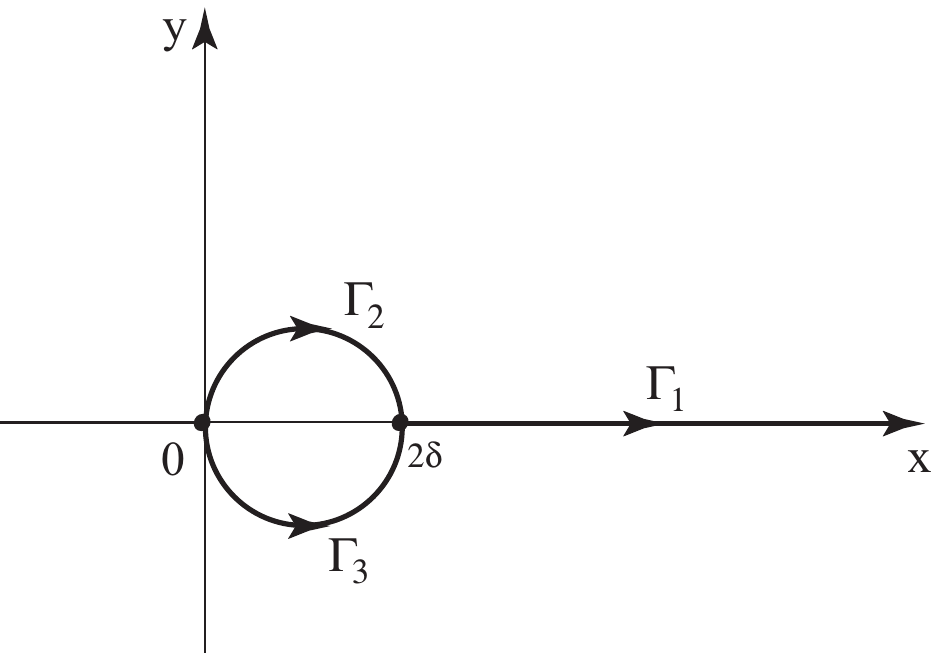} \end{center}
  \caption{Contour of orthogonality, $\Gamma=\Gamma_1\cup\Gamma_2\cup\Gamma_3$. }
 \label{contour-ortho}
\end{figure}

{  With the weight function $w_t(z)$ given  in \eqref{pL new weight}, the corresponding Hankel determinant $D_n[w;t]$ in \eqref{Hankel} is well-defined. However, since  $w_t(z)$ is not positive on $\Gamma$, the orthogonal polynomials $\pi_{k,n}(z)$ in \eqref{pL polynomials: orthogonality:complex} may not exist for some $k$, and \eqref{D-n and leading coeff}   only makes sense if all polynomials $\pi_{j,n}$ for $j=0,1,\cdots,k-1$      exist.
It is worth mentioning that  as part of our results, we will show that there exists a $t_{cr}<0$, such that $\pi_{n,n}(z)$ exists for $n$ large enough and $t\geq t_{cr}$; cf. Section \ref{sec-rhp-OPs}.  The}
        recurrence relation \eqref{recurrence relation} still makes sense for such $t$ if all of $\pi_{k-1,n}(x)$, $\pi_{k,n}(x)$ and $\pi_{k+1,n}(x)$ exist.   Note that   in the literature,  the polynomials with non-Hermitian orthogonality have been studied in several different contexts; see for example \cite{Ber:Tov,Ble:Dea,Dui:Kuij,fik1992},  where the cubic and quartic potentials are considered.

One of the main motivations of this paper comes from the Wigner time-delay in the study of quantum mechanical scattering problem.  To describe the electronic transport in mesoscopic (coherent) conductors, Wigner \cite{Wigner1955} introduced the so-called time-delay matrix $Q$; see also Eisenbud \cite{Eisenbud1948} and Smith \cite{Smith1960}. The eigenvalues $\tau_k$ of $Q$, called the proper delay times, are used to describe the time-dependence of a scattering process. The joint distribution  of the inverse proper delay time $\gamma_k=1/\tau_k$ was found, by Brouwer et al.\;\cite{Bro:Fra:Bee}, to be
\begin{equation}\label{jpd of inverse proper time delay}
P(\gamma_1,\gamma_2,...\gamma_n)=\frac 1{Z_n} \prod^n_{i=1}\gamma_i^n e^{-\tau_H\gamma_i}\prod_{1\leq i<j\leq n}|\gamma_i-\gamma_j|^2.
\end{equation}
Then the probability density function of the average of the proper time delay, namely the Wigner time-delay distribution, is defined as
\begin{equation}\label{wigner time delay distribution}
P_n(\tau)=\frac 1{Z_n}\int_{\mathbb{R}_+^{n}}  \prod^n_{i=1}\gamma_i^n e^{-n\gamma_i}\prod_{1\leq i<j\leq n}|\gamma_i-\gamma_j|^2\; \delta\left (\tau-\sum_1^n\frac 1{\gamma_i}\right )\prod_{i=1}^n d\gamma_i.
\end{equation}
The moment generating function is the Laplace transformation of the Wigner time-delay distribution
\begin{equation}\label{moment generated function}
M_n(z)=\int_0^\infty e^{-z\tau}P_n(\tau)d\tau= \frac 1{Z_n}\int_{\mathbb{R}_+^{n}} \prod^n_{i=1}\gamma_i^n e^{-n\gamma_i-\frac z {\gamma_i}}\prod_{1\leq i<j\leq n}|\gamma_i-\gamma_j|^2 \prod_{i=1}^n d\gamma_i,
\end{equation}
which is closely related to the Hankel determinant \eqref{Hankel} as follows:
\begin{equation}\label{moment generated function and Hankel}
M_n(nt)=\frac {D_n[w(x;t)]}{D_n[w(x;0)]}.
\end{equation}
Recently, Texier and Majumdar \cite{texier} studied the  Wigner time-delay distribution by using  a Coulomb gas method. They showed  that
\begin{equation}
  P(\gamma_1, \gamma_2, \cdots, \gamma_n) \sim \exp\{-n^2 {E}[\rho(x)]\}\qquad \textrm{for large  } n,
\end{equation}
where $\rho(x)dx$ is the unique minimizer for an energy problem with the external field $V_t(x)$ in \eqref{potential}, and $E[\rho(x)]$ is the minimum energy. Moreover, the density $\rho(x)$ is computed explicitly in \cite{texier}, namely,
\begin{equation}\label{e measure}
\rho(x)=\frac {x+c}{2\pi x^2}\sqrt{(x-a)(b-x)}~~\mbox{for}~x\in [a, b],~ \mbox{with}~0<a<b,~c=t/\sqrt{ab}.
\end{equation}
{  Here  positive $a$ and $b$ are independent of $x$ and  implicitly determined by $t$ as follows:
\begin{equation}\label{a,b equation}
        1+\frac t{2ab}(a+b)=\sqrt{ab}; \quad \frac 12(a+b)-\frac t{\sqrt{ab}}=3.
  \end{equation}
  One may notice that  $\rho(x)dx$ is a probability measure on $[a,b]$ as long as  $a+c$  is  non-negative. Since $a+c$ is a continuous function of $t$, we see that $\rho(x)$ in \eqref{e measure} is non-negative for $t>t_{cr}$, where $t_{cr}= -\frac 3 4 \left (2^{1/3}-1\right )^2$ is the critical value of $t$ corresponding to the case $a+c=0$; see Theorem \ref{Theorem: Asymptotic of Hankel}. }
 It is very interesting to  observe  that, for {  this}  $t_{cr}<0$, we have $c_{cr} = -a_{cr} < 0$ and
\begin{equation}\label{e measure-critical}
\rho(x)=\frac 1{2\pi x^2}(x-a_{cr})^{3/2}(b_{cr}-x)^{1/2},
\end{equation}
where a phase transition emerges at the left endpoint $x=a_{cr}$. Here the critical values $t_{cr}$, $a_{cr}$ and $b_{cr}$ are explicitly given in \eqref{t-critical} and \eqref{a,b, mu critical}.


It is also interesting to look at our problem from another point of view. Due to the term $\frac{t}{x}$ in the exponent of \eqref{pL weight}, we may consider the  origin as an essential singular point of the weight function. In recent years, orthogonal polynomials whose weights possess essential singularities have been studied extensively. For example, Chen and Its \cite{ci} consider orthogonal polynomials associated with the weight
\begin{equation} \label{chen-its-weight}
  w_t(x)=x^{\alpha}e^{-x-\frac tx},\quad x\in(0,\infty),~\alpha>0 \textrm{ and }t>0.
\end{equation}
They show  that, for fixed degree $n$, the recurrence coefficient satisfies a particular Painlev\'e III equation with respect to the parameter $t$, and the Hankel determinant of fixed size $D_n[w_t(x)]$ equals to the isomonodromy  $\tau$-function of the Painlev\'e III equation with parameters depending on $n$.
The matrix model and Hankel determinants $D_n[w;t]$ associated with the weight   in \eqref{chen-its-weight}
were also encountered      by  Osipov and Kanzieper \cite{Osi:Kan} in bosonic replica field theories.
 Later, the large $n$ asymptotics of the Hankel determinants $D_n[w_t(x)]$ associated with the weight function in \eqref{chen-its-weight} is studied by the current authors in \cite{xdz2014} and \cite{xdz2015}. For $t\in(0,d]$,  the asymptotics of the
Hankel determinants are derived and expressed in terms of certain Painlev\'e III transcendents. The asymptotics of the recurrence coefficients are also obtained therein. In the case of the Gaussian weight perturbed by essential singularity
\begin{equation}
  e^{-x^2-t/x^2},\quad x\in \mathbb{R},
\end{equation}
the double scaling limit of the Hankel determinants are also characterized in terms of  Painlev\'e III transcendents by Brightmore et al.\;in \cite{bmm}. Recently, Atkin, Claeys and Mezzadri \cite{At:Cla:Mez} extend the results to the case of Laguerre and Gaussian weight perturbed by a pole of higher order  at the
origin, they obtain the double scaling asymptotics of the Hankel determinants in terms of   a hierarchy of higher order analogs to the Painlev\'e III equation.

The main objective of this paper is to study the Hankel determinant  $D_k[w;t]$ with respect to the weight \eqref{pL new weight} in the region $t<0$. First, for fixed degree $k$, we will show  that the recurrence coefficient $\alpha_{k,n}$ satisfies a Painlev\'{e} III equation, and the Hankel determinant  $D_k[w;t]$ equals to the isomonodromy $\tau$-function of the Painlev\'{e} III equation.  Then, we will derive   the double scaling limit of the Hankel determinant $D_n[w;t]$, the recurrence coefficients and leading coefficients of the associated orthogonal polynomials.  Our results are described   in terms of a certain tronqu\'{e}e solution of the Painlev\'{e} I equation.


\subsection{A model Riemann-Hilbert problem for Painlev\'{e} I} \label{sec-model Riemann-Hilbert problem}

To state our results, we need certain special solutions to the Painlev\'e I equation
\begin{equation} \label{P1-eqn}
  y''(s) = 6y^2(s)  + s.
\end{equation}
The reader is referred to \cite[Ch.32]{nist}
for properties of the Painlev\'e I equation, as well as the other Painlev\'e equations.   In \cite{Kapaev}, Kapaev formulates the following model  Riemann-Hilbert (RH, for short)  problem for $\Psi(\zeta)= \Psi(\zeta; s)$, associated with the Painlev\'{e} I equation. This model RH problem will play a crucial role later in the construction of a local parametrix in the steepest descent analysis.

\begin{itemize}

  \item[(a)] $\Psi(\zeta; s)$ is analytic for $\zeta \in \mathbb{C} \setminus \Gamma_\Psi$, where
  \begin{equation}
    \Gamma_\Psi= \gamma_{-2} \cup \gamma_{-1} \cup \gamma_{1} \cup \gamma_{2} \cup \gamma^*
  \end{equation}
  are illustrated in Figure \ref{contour-p1-model}.

  \begin{figure}[h]
 \begin{center}
   \includegraphics[width=7.5cm]{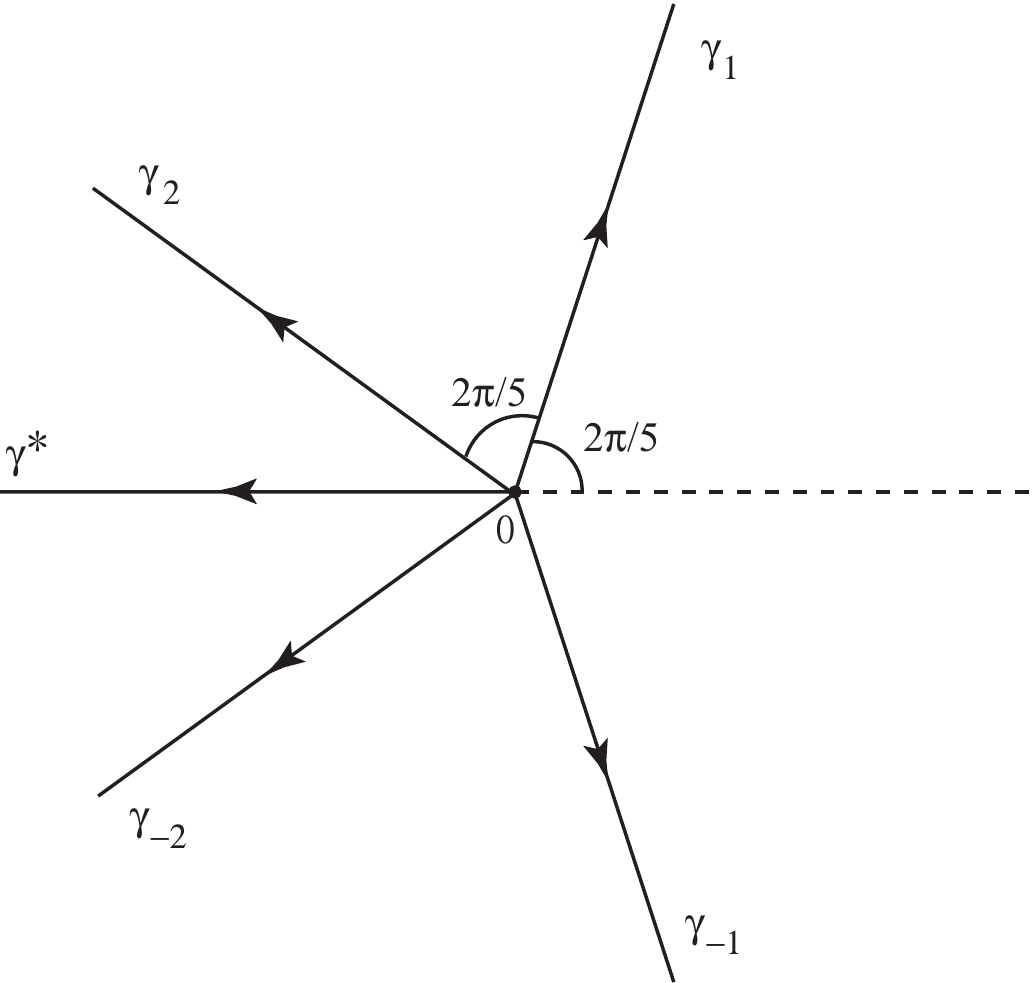} \end{center}
  \caption{The contour $\Gamma_\Psi$ associated with the Painlev\'e I equation}
 \label{contour-p1-model}
\end{figure}

  \item[(b)] Let $\Psi_{\pm} (\zeta; s)$ denote the limiting values of $\Psi(\zeta; s)$ as $\zeta$ tends to the contour $\Gamma_\Psi$ from the left and right sides, respectively. Then, $\Psi(\zeta; s)$ satisfies the following jump conditions
   \begin{equation}\label{Psi--jump}
  \Psi_+(\zeta; s)= \Psi_-(\zeta; s)
 \left\{\begin{array}{ll}
          \left(
                               \begin{array}{cc}
                                 1 & s_i \\
                               0& 1 \\
                                 \end{array}
                             \right), & z \in \gamma_k, \ k=\pm1; \\[.4cm]
           \left(
                               \begin{array}{cc}
                                 1 & 0 \\
                                 i& 1 \\
                                 \end{array}
                             \right), & z \in \gamma_k, \ k=\pm 2;  \\[.4cm]
          \left(
                               \begin{array}{cc}
                                0 &-i\\
                               -i&0 \\
                                 \end{array}
                             \right),  &  z\in \gamma^*,
        \end{array}\right .
 \end{equation}
        where $s_1=(1-\alpha)i$   and $s_{-1}=\alpha i$, with $\alpha$ being a complex constant.

   \item[(c)] As $\zeta \to \infty$, $\Psi(\zeta; s)$ satisfies the asymptotic condition
  \begin{equation}\label{Psi-infty}
    \Psi(\zeta; s)= \zeta^{\frac{1}{4} \sigma_3}  \frac{\sigma_3+\sigma_1}{\sqrt{2}}\left(I + \frac{\Psi_{-1}(s)}{\sqrt{\zeta}} + \frac{\Psi_{-2}(s)}{\zeta} + O(\zeta^{-\frac 32})\right)  e^{\theta(\zeta,s) \sigma_3}
    \end{equation}
    for $\arg \zeta\in (-\pi, \pi)$, where
    \begin{equation} \label{theta-def}
      \theta(\zeta,s)=\frac 45 \zeta^{\frac 52}+s\zeta^{\frac 12},
    \end{equation}
    $\sigma_1$ and $\sigma_3$ are the Pauli matrices
    \begin{equation*}
      \sigma_1 = \begin{pmatrix}
        0 & 1 \\ 1 & 0
      \end{pmatrix}, \qquad \sigma_3 = \begin{pmatrix}
        1 & 0 \\ 0 & -1
      \end{pmatrix}.
    \end{equation*}

\end{itemize}

It is known that, for each $\alpha\in \mathbb{C}$,
\begin{equation} \label{tronquee-sol}
  y_\alpha(s) = 2(\Psi_{-2} (s))_{12}
\end{equation}
is a solution of the Painlev\'e I equation \eqref{P1-eqn}. {  As a consequence, the above RH problem for $\Psi(\zeta; s)$ has a solution if and only if $s$ is not a pole of $y_\alpha(s)$. Due to the meromorphic property of the Painlev\'e I transcendents, one also see that the solution of the above RH problem for  $\Psi(\zeta; s)$ is meromorphic in the parameter $s$.
Moreover, it is shown in Kapaev \cite{Kapaev}}   that $y_\alpha(s)$ is the so-called \emph{tronqu\'ee} solution of Painlev\'e I whose asymptotic behavior is given by
\begin{equation}
y_\alpha(z)=y_0(z)+\frac {\alpha i}{\sqrt{\pi}}2^{-\frac {11}8}(-3z)^{-\frac 18}\exp\left[-\frac 152^{\frac {11}4}3^{\frac 14}(-z)^{\frac 54}\right]\left (1+O(z^{-\frac 38})\right )
\end{equation}
as $z\to \infty$ and $\arg z =\arg(-z)+\pi    \in[\frac 35\pi,\pi]$. Here $y_0(z)$ is the \emph{tritronqu\'ee} solution satisfying
\begin{equation}\label{Painleve I}
y_0(z)\sim\sqrt{-z/6}\left[ 1+ \sum_{k=1}^{\infty}a_k(-z)^{-5k/2} \right] \quad \textrm{as } z \to \infty, \ -\frac{\pi}{5}<\arg z< \frac{7\pi}{5},
\end{equation}
where  the coefficients $a_k$ can be determined recursively; see for example Joshi and Kitaev \cite{Jos:Kit}. The solution $y_\alpha(z)$ will appear in our main results below.

We mention several known facts about the coefficients in \eqref{Psi-infty} in addition to   \eqref{tronquee-sol}.  For example, the explicit formulas of $\Psi_{-1}(s)$ and $\Psi_{-2}(s)$   are given in \cite{Kapaev} {  as}
\begin{equation}
  \Psi_{-1}(s)= -\mathcal{H}_\alpha(s)\sigma_3, \qquad \Psi_{-2}(s) = \frac 12 \biggl( \mathcal{H}_\alpha^2(s)I+y_\alpha(s)\sigma_1 \biggr),
\end{equation}
where
\begin{equation} \label{P1-Hamilton}
       \mathcal{H}_\alpha(s)=\frac 12 y_\alpha'^2(s)-2y_\alpha^3(s)-sy_\alpha(s)
\end{equation}
is the Hamiltonian of Painlev\'{e} I.

\subsection{Statement of main results}

First of all, when the degree $k$ is fixed, we show that the recurrence coefficient $\alpha_{k,n}(t)$ satisfies a particular Painlev\'{e} III equation with certain initial conditions. Moreover, we prove that  the Hankel determinant $D_k[w(z;t)]$ is   related to the $\tau$-function of the Painlev\'{e} III equation. Similar results for the weight in \eqref{chen-its-weight} {  have} been obtained  by Chen and Its \cite{ci}.

\begin{thm}\label{Theorem: Hankel as tau function}
For fixed non-negative integer $k$, let $\alpha_{k,n}$ be the recurrence coefficient in \eqref{recurrence relation}, and
\begin{equation}
a_{k,n}(t)=\alpha_{k,n}(t)-\frac {2k+1+n}{n}.
\end{equation}
Then $a_k(t)= a_{k,n}(t)$ satisfies the following Painlev\'e III equation
\begin{equation}\label{introduction: painleve III}
a_k''=\frac {(a_k')^2}{a_k}-\frac {a_k'}{t}+n(2k+1+n)\frac {a_k^2}{t^2}+\frac {n^2a_k^3}{t^2}+\frac{n^2}{t}-\frac{n^2}{a_k},
\end{equation}
with the initial conditions $a_{k}(0)=0$, $a_{k}'(0)=1$. Moreover, we have
\begin{equation}\label{introduction:painleve III-tao function}
D_k[w;t]=\mathrm{const} \cdot \tau(t)\; e^{  {n^2t}/{2}}\; t^{  {k(k+n)}/{2}},
\end{equation}
where $\tau(t)$ is the Jimbo-Miwa-Ueno isomonodromy $\tau$-function of the above Painlev\'{e} III equation.
\end {thm}

Next, we let $k=n$ and consider the double scaling limit when $n \to \infty$ and $t \to t_{cr}$ simultaneously. We show  that the asymptotics of the Hankel determinant  $D_n[w(z;t)]$ associated with the weight in \eqref{pL polynomials: orthogonality:complex} can be expressed in terms of the tronqu\'{e}e solution $y_\alpha(s)$ of Painlev\'{e} I equation given in \eqref{tronquee-sol}.

\begin{thm}\label{Theorem: Asymptotic of Hankel}
Let the constants $t_{cr}$, $a_{cr}$ and $b_{cr}$ be defined as
 \begin{equation} \label{t-critical}
    t_{cr}=-\frac 34(2^{1/3}-1)^2\approx -0.051
  \end{equation}and
  \begin{equation} \label{a,b, mu critical}
  a_{cr} = \frac{1}{2} (3-2^{1/3} - 2^{2/3})\approx 0.076, \quad b_{cr} = \frac{3}{2} (1 + 2^{1/3} + 2^{2/3})\approx 5.771.
 \end{equation}
For $n\to \infty$ and $t\to t_{cr}$ in a way such that
\begin{equation} \label{s-t-double scaling}
  s^*= \left\{  (2a_{cr})^{-\frac 35}(a_{cr} b_{cr})^{-\frac 12}(b_{cr}-a_{cr})^{\frac 25}\right \}          n^{\frac 45}(t_{cr}-t)
\end{equation}
remains bounded. Suppose $\alpha \in \mathbb{C}$ is fixed and $s^*$ is not a pole of the tronqu\'{e}e solution $y_\alpha(s)$, then an  asymptotic approximation of the logarithmic derivative of the Hankel determinant
$H_{n,n}=t \frac {d}{dt}\log D_n[w; t]$ associated with the weight function \eqref{pL new weight}
 is given by
\begin{equation}\label{introduction: asymptotic of  H}
\frac {d}{dt}H_{n,n}(t)=-\frac{ n^2} 4 \left(\sqrt{\frac {a_{cr}}{b_{cr}}}+\sqrt{\frac {b_{cr}}{a_{cr}}}-2 -\frac 1{n^{\frac 25}} \frac {2(b_{cr}-a_{cr})^{\frac 45}}{(2a_{cr})^{\frac 15}\sqrt{a_{cr}b_{cr}}}y_\alpha(s^*)+O\left (\frac 1 {n^{\frac 3 5}}\right )\right).
\end{equation}
\end{thm}\vskip .5cm

We would also derive the double scaling limit of the recurrence coefficients and the leading coefficients of the orthonormal polynomials.

\begin{thm}\label{Theorem: Asymptotic of recurrence coff}
Under the same conditions as in the previous theorem,
         {  the monic polynomial $\pi_{n,n}(z)$ defined in \eqref{pL polynomials: orthogonality:complex} exists for large enough $n$   and $t$ close to $t_{cr}$}. Moreover,
we have the asymptotics of the recurrence coefficients
\begin{equation}\label{introduction:a-n asymtotics}
   a_{n,n}=\frac {t}{\sqrt{a_{cr}b_{cr}}}\left(1 -\frac{2^{4/5} y_\alpha(s^*)}{a_{cr}^{1/5}(b_{cr}-a_{cr})^{1/5}n^{2/5}} +O\left (\frac 1 {n^{3/5}}\right )\right),
\end{equation}
\begin{equation}\label{introduction:beta asymtotics}
   \beta_{n,n}=\frac {(b_{cr}-a_{cr})^2}{16}-\frac {(2a_{cr}(b_{cr}-a_{cr}))^{4/5}y_\alpha(s^*)}{4}\frac {1}{n^{2/5}}+O\left (\frac 1 {n^{3/5}}\right )
\end{equation}
and
\begin{equation}\label{leading coefficient asymtotics-one term}
   \gamma_{n,n}^2=\frac {2}{\pi (b_{cr}-a_{cr})}e^{-nl}\left(1+\frac {2(2a_{cr})^{4/5}\mathcal{H}_\alpha(s^*)}{(b_{cr}-a_{cr})^{1/5}n^{1/5}}+O\left (\frac 1 {n^{2/5}}\right )\right),
\end{equation}
where $\mathcal{H}_\alpha(s)$ is the Hamiltonian of Painlev\'{e} I given in \eqref{P1-Hamilton}.
\end{thm}

\begin{rmk}
  It is well-known that the tronqu\'ee solutions of Painlev\'e I are meromorphic functions and possess infinitely many poles in the complex plane. Therefore, to make the results valid in the above theorems, we require the $s^*$ in \eqref{s-t-double scaling} is bounded away from the poles of $y_\alpha(s)$. Recently, through a more delicate \emph{triple scaling limit}, Bertola and Tovbis \cite{Ber:Tov} successfully obtain the asymptotics near the poles of $y_\alpha(s)$.
  Similar results near the poles of $y_\alpha(s)$ might be derived by using their ideas in \cite{Ber:Tov}. However, we do not pursue that part. Instead, we focus on  the main task of the present  paper  to demonstrate that the Painlev\'e I asymptotics can also occur for the weight \eqref{pL weight} with negative $t$.
\end{rmk}

The rest of the paper is arranged as follows. In Section \ref{finite determinants}, we provide a RH problem for the  orthogonal polynomials with respect to the weight \eqref{pL new weight}. A transformed version of the solution is shown to fulfill a Lax pair, which is closely related to the Painlev\'{e} III equation. Several differential identities are stated and justified.  Theorem \ref{Theorem: Hankel as tau function} is also proved in this section.     Section \ref{sec-E-measure} is devoted to the determination of equilibrium measures, involving  a positive measure and a signed measure. In Section \ref{sec-RH-analysis},
we carry out  a nonlinear steepest descent analysis of the RH problem for the orthogonal polynomials. Particular attention will be paid to  the construction of the local parametrix at the critical endpoint $z=a_{cr}$, where the Painlev\'e I transcendents are involved.  Then,  the proofs of Theorems \ref{Theorem: Asymptotic of Hankel} and \ref{Theorem: Asymptotic of recurrence coff} are given in the last section, Section \ref{sec-proofs}.


\section{Finite Hankel determinants and Painlev\'{e} III equation} \label{finite determinants}

In this section, we state the RH  problem for the perturbed Laguerre orthogonal polynomials. Then we show that after some elementary transformations, the RH problem is transformed into a RH problem for the Painlev\'{e} III equation.  As a consequence, we  derive a Painlev\'{e} III equation satisfied by the recurrence coefficient $\alpha_{k,n}$ up to a  translation,   and establish a relation between the finite Hankel determinant of the perturbed Laguerre weight in \eqref{pL polynomials: orthogonality:complex} with the $\tau$-function of this  Painlev\'{e} III equation. Several differential identities for the Hankel determinants and the recurrence coefficients of the perturbed Laguerre orthogonal polynomials are also derived. The identities  are important in the asymptotic analysis  in  later sections. Although our calculations are similar to those in Chen and Its \cite{ci}, we think it is convenient for the reader to have more details.

\subsection{Riemann-Hilbert problem for orthogonal polynomials and differential identities} \label{sec-rhp-OPs}

We state the   RH  problem for the perturbed Laguerre orthogonal polynomials as follows:
\begin{itemize}
  \item[(Y1)]  $Y(z)$ is analytic in
  $\mathbb{C}\backslash \Gamma_j$, $j=1,2,3$; see Figure \ref{contour-ortho};

  \item[(Y2)] $Y(z)$  satisfies the jump condition
  \begin{equation}\label{Y-jump}
  Y_+(z)=Y_-(z) \left(
                               \begin{array}{cc}
                                 1 & w(z; t) \\
                                 0 & 1 \\
                                 \end{array}
                             \right),
    \qquad z\in \Gamma= \Gamma_1\cup \Gamma_2\cup\Gamma_3,
    \end{equation}
    where $w(z; t)$ is the weight function piecewise-defined in \eqref{pL new weight};

  \item[(Y3)]  The asymptotic behavior of $Y(z)$  at infinity is
  \begin{equation}\label{Y-infty}
  Y(z)=\left (I+O\left (  1 /z\right )\right )\left(
                               \begin{array}{cc}
                                 z^k & 0 \\
                                 0 & z^{-k} \\
                               \end{array}
                             \right),\quad \mbox{as}\quad z\rightarrow
                             \infty ;\end{equation}
{  \item[(Y4)]  As $z\to 0$, $Y(z) = O(1)$.}
  \end{itemize}

Using a by now standard argument, originally due to Fokas, Its, and Kitaev \cite{fik1992}, the solution of the above RH problem, if it exists, is uniquely given by
\begin{equation}\label{Y-solution}
Y(z)= \left (\begin{array}{cc}
\pi_k(z)& \frac 1 {2\pi i}
 \int_{\Gamma} \frac {\pi_k(s) w(s; t) }{s-z} ds\\[0.2cm]
-2\pi i \gamma_{k-1}^2 \;\pi_{k-1}(z)& -   \gamma_{k-1}^2\;
 \int_{\Gamma} \frac {\pi_{k-1}(s) w(s; t) }{s-z} ds \end{array} \right ),
\end{equation}
where $\pi_k(z)=\pi_{k, n}(z)$ is the monic perturbed Laguerre orthogonal polynomials defined in \eqref{pL polynomials: orthogonality:complex} and $\gamma_k=\gamma_{k,n}(t)$  is the leading coefficient  for the
orthonormal  polynomial  of degree $k$.
        {  To show the existence of $\pi_{n, n}(z)$ when $n$ is large enough, we will apply a series of invertible transformations to transform the original RH problem $Y$ to a new RH problem for $R$, which is solvable for sufficiently large $n$, $t$ close to $t_{cr}$ as in \eqref{s-t-double scaling}, and $s^*$ is not a pole of the tronqu\'{e}e solution $y_\alpha(s)$. Tracing back the invertible transformations,}  we will see that the RH problem is solvable {  under the same conditions. Indeed, it is also possible to prove the solvability for $t\geq t_{cr}$. However, since we are interested in the phase transition near $t_{cr}$, we don't consider the case when $t_{cr} < t < 0$ in the subsequent analysis.}   Thus,  the perturbed Laguerre orthogonal polynomials are well-defined for $k=n$ large enough.

Next, we derive some differential identities  for the recurrence coefficients and the logarithmic derivative of the Hankel determinant associated
with the perturbed Laguerre weight $w(z)=w(z; t)$ in  \eqref{pL new weight}. The results are expressed    in terms of the entries of $Y(z)$.

\begin{lem}\label{lem:differential identity}{  Assume that $t>0$.}
Let $\alpha_{k,n}(t)$ and $\beta_{k,n}(t)$ be the recurrence coefficients in \eqref{recurrence relation}, and $D_k[w;t]$ be the Hankel determinant in \eqref{Hankel}. Define
\begin{equation}\label{a-n and alpha-n}
      a_{k,n}(t)= \alpha_{k,n}(t)-\frac 1n(2k+1+n)
\end{equation}
and
\begin{equation}\label{logrithmic derivative of hankel}
     H_{k,n}(t)=t \frac {d}{dt} \log D_k[w;t].
\end{equation}
Then we have
\begin{equation} \label{a-n in terms of  Y}
   a_{k,n}(t)=t\gamma_{k,n}^2\int_\Gamma\frac{\pi _{k,n}^2(z)w(z; t)}{z}dz=2\pi i t\gamma_{k,n}^2(t) Y_{11}(0)Y_{12}(0),
\end{equation}
\begin{equation}\label{beta and H}
 \beta_{k,n}(t)=\frac 1{n^2} \biggl[ k(k+n)+ t\frac d{dt} H_{k,n}(t)-H_{k,n}(t) \biggr]
\end{equation}
and
\begin{equation}\label{differential identity}
    \frac d{dt}H_{k,n}(t)=n^2\gamma_{k-1,n}^2\int_\Gamma\frac{\pi _{k-1,n}(z)\pi _{k,n}(z)w(z; t)}{z}dz=-n^2 Y_{12}(0)Y_{21}(0).
\end{equation}

\end{lem}

\begin{proof}{  Since $t>0$, the   orthogonal polynomials $\pi_{k,n}$ exist for all nonnegative $k$ and positive $n$.}   First,
we consider the recurrence coefficient $\alpha_{k,n}(t)$. Based on the three-term recurrence relation \eqref{recurrence relation} and the orthogonality condition \eqref{pL polynomials: orthogonality:complex}, we get
\begin{equation}
      \alpha_{k,n}(t)=\gamma_{k,n}^2(t)\int_\Gamma z\pi_{k,n}^2 (z) w(z)dz.
\end{equation}
Using the fact that $w(z) = \frac{w(z)}{z} + \frac{tw(z)}{z^2} - \frac{w'(z)}{n}$ and integrating  by part once, the above formula gives us
\begin{equation} \label{akn-integral}
      a_{k,n}(t)=t\gamma_{k,n}^2(t)\int_\Gamma\frac{\pi _{k,n}^2(z)w(z)}{z}dz.
\end{equation}
Then \eqref{a-n in terms of  Y} follows from a partial fraction decomposition of {  $\frac {\pi_{k, n}(z)}z$}, the orthogonality condition \eqref{pL polynomials: orthogonality:complex},  and the explicit expression of $Y(z)$ in \eqref{Y-solution}.

Next, we consider the Hankel determinant. Recall that  the Hankel determinant can be expressed in terms of the leading coefficients as
\begin{equation*}
  D_k[w;t] = \prod_{j=0}^{k-1} \gamma_{j,n}^{-2}(t);
\end{equation*}see \eqref{D-n and leading coeff}.
Taking logarithmic derivative of both sides of the above equation with respect to $t$ and using the integral representation of the leading
coefficients in \eqref{pL polynomials: orthogonality:complex}, we get
 \begin{equation}\label{H and a-n}
 H_{k,n}(t)=-n\sum_{j=0}^{k-1}a_{j,n}(t).
\end{equation}
Differentiating the above formula again,  we get from \eqref {a-n and alpha-n}
 \begin{equation}\label{H derivative and alpha}
    \frac d{dt}H_{k,n}(t)=-n \frac d{dt}\sum_{j=0}^{k-1}\alpha_{j,n}(t).
   \end{equation}
Let $\mathfrak{p}_{k,n}(t)$ be the  coefficient of the $z^{k-1}$ term in $\pi_{k,n}(z)$, i.e.,
\begin{equation} \label{pkn-formula}
  \pi_{k,n}(z)=z^k+\mathfrak{p}_{k,n}(t)z^{k-1}+ \cdots.
\end{equation}
Comparing the $x^k$ powers in the recurrence relation \eqref{recurrence relation}, we obtain
\begin{equation}\label{p and alpha-n}
   \alpha_{k,n}=\mathfrak{p}_{k,n}(t)-\mathfrak{p}_{k+1,n}(t).
\end{equation}
To derive $\frac d{dt}H_{k,n}(t)$, one can see from \eqref{H derivative and alpha} and \eqref{p and alpha-n} that it is sufficient to obtain $\frac d{dt} \mathfrak{p}_{k,n}(t)$. This can be done by taking derivative of the following orthogonal formula with respect to the parameter $t$
\begin{equation*}
  \int_{\Gamma}\pi_{k,n}(z)\pi_{k-1,n}(z)w(z)dz=0.
\end{equation*}
More precisely, taking into account the orthogonal relation \eqref{pL polynomials: orthogonality:complex} and the fact that $\frac {\partial w(z; t)} {\partial t}=-\frac n z w(z; t)$,  we have
 \begin{equation}\label{p derivative}
 \frac d{dt} \mathfrak{p}_{k,n}(t)=n\gamma_{k-1,n}^2\int_{\Gamma} \frac {\pi_{k,n}(z)\pi_{k-1,n}(z)w(z)}{z}dz.
 \end{equation}
Then, \eqref{differential identity} follows from a combination of \eqref{H derivative and alpha}, \eqref{p and alpha-n} and  \eqref{p derivative}, as well as the definition of $Y(z)$ in \eqref{Y-solution}.

Finally, let us study $\beta_{k,n}(t)$. Using the  ideas leading to \eqref{akn-integral}, we have
\begin{equation}\label{beta integrate by part}
     \beta_{k,n}(t)=t\gamma_{k-1,n}^2\int_{\Gamma} \frac {\pi_{k,n}(z)\pi_{k-1,n}(z)w(z)}{z}dz-\frac  1n \mathfrak{p}_{k,n}(t),
   \end{equation}
where $\mathfrak{p}_{k,n}(t)$ is introduced  in \eqref{pkn-formula}. The first term on the right-hand side is $\frac{t}{n^2}\frac d{dt} H_{k,n}(t)$; cf. \eqref{differential identity}. An expression for the term on the extreme right can be obtained by deriving   $\mathfrak{p}_{k,n}(t)=\sum_{j=0}^{k-1} \alpha_{j,n}$ from  \eqref{p and alpha-n},  and using \eqref{a-n and alpha-n} and \eqref{H and a-n}.

This completes the proof of our lemma.
\end{proof}

 {
 \begin{rmk}
For later use, we need the differential identities   of   Lemma \ref{lem:differential identity}   in the case when $k=n$ is large and   $t\sim t_{cr}$.
  They can be obtained through
 an analytic continuation argument.
Indeed, $Y(z)$ determined by RH problem  exists in this case, and  is related to the $\Psi$-function of the third Painlev\'{e} equation after an elementary transformation given in \eqref{Psi function-painleve III}.
Thus $Y(z)$ is meromorphic with respect to $t$ in the cut plane $\arg(-t) \in (- \pi/2, 3\pi/2)$. In particular, both $Y$ and the Hankel determinant are analytic in a domain
containing the interval  $t>0$ and a neighborhood of $t=t_{cr}$.   Note  that the identities \eqref{a-n in terms of  Y}-\eqref{differential identity} hold for $t>0$, then, by analytic continuation, they also hold
for $t$ close to $t_{cr}$.
We conclude that for  $k=n$   large and   $t\sim t_{cr}$, the identities  \eqref{a-n in terms of  Y}-\eqref{differential identity} are also true.
Similar
argument
has previously been used in
  Bleher and Dea\~{n}o \cite{Ble:Dea1}.
\end{rmk}  }

\subsection{Relation to the Painlev\'{e} III equation} \label{subsec-painleve III}

Introduce a purely imaginary parameter $s= n  i \sqrt{-t}$, and define
\begin{equation}\label{Psi function-painleve III}
\Phi(\lambda,s)=\left (\frac {n i}s\right )^{(\frac n2 +k)\sigma_3}Y\left (\frac {s\lambda}{n i}\right )e^{\frac i2 (s\lambda-\frac s{\lambda}) \sigma_3}\left (\frac {s\lambda}{n i}\right )^{\frac n2\sigma_3}, ~~\lambda\not\in
\Gamma^*, \end{equation} where $\Gamma^*= \Gamma^*_1 \cup    \Gamma^*_2 \cup  \Gamma^*_3  = \frac{1}{\sqrt{-t}} \Gamma$ is the rescaled contour.
Then, $\Phi(\lambda) =\Phi(\lambda,s)$ solves   the following RH problem with constant jumps:
\begin{itemize}
  \item[(i)] $\Phi(\lambda)$ is analytic for $\lambda \in \mathbb{C} \setminus \cup_{j=1}^3\Gamma^*_j$.
  As $\Gamma^*$ and $\Gamma$ only differ by a scale, one may refer to    Figure \ref{contour-ortho}      to see the properties  of the contour $\Gamma^*$.

  \item[(ii)] $\Phi(\lambda)$  satisfies the jump condition
  \begin{equation}\label{Phi-Painleve iii-jump}
  \Phi_+(\lambda)=\Phi_-(\lambda) \left(
                               \begin{array}{cc}
                                 1 & c_j \\
                                 0 & 1 \\
                                 \end{array}
                             \right),
    \quad \lambda \in \Gamma^*_j,~j=1,2,3,
    \end{equation}
    where $c_1=1,c_2=\alpha,c_3=1-\alpha$; cf. \eqref{pL new weight}.

  \item[(iii)] The asymptotic behavior of $\Phi(\lambda)$ at infinity is
  \begin{equation}\label{Phi-Painleve III-infty}
  \Phi(\lambda)=\left (I+\sum_{k=1}^{\infty}\frac {\Phi_{-k}}{\lambda^k}\right )\lambda^{(\frac n2+k)\sigma_3}e^{\frac {is\lambda}2\sigma_3} \quad \mbox{as}\quad\lambda\rightarrow \infty ,
  \end{equation}
  where
  \begin{equation}\label{Phi-Painleve III-infty-coeff}
  \Phi_{-1}=\frac{n i }s \left(\frac {n i}s\right)^{(\frac n2+k)\sigma_3}\left(
  \begin{array}{cc}
  \mathfrak{p}_{k,n}(t) -\frac {s^2}{2n}& -\frac {1}{2\pi i\gamma_{k,n}(t)^2} \\
     -2\pi i \gamma_{k-1,n}(t)^2 & - \mathfrak{p}_{k,n}(t)+\frac {s^2}{2n} \\
   \end{array}
   \right)
   \left(\frac{ n i}s\right)^{-(\frac n2+k)\sigma_3},
   \end{equation}
   In the above formula, $\gamma_{k,n}$ and $\mathfrak{p}_{k,n}$ are, respectively,    the leading coefficient  of the $k$-th orthonormal polynomial,   and the
    coefficient  of the
    $z^{k-1}$ term in the
    $k$-th  monic orthogonal polynomial introduced in \eqref{pkn-formula},   with respect  to   the varying perturbed Laguerre weight in \eqref{pL new weight} and \eqref{pL polynomials: orthogonality:complex}.

\item[(iv)] The asymptotic behavior of $\Phi(\lambda)$ at $\lambda=0$ is
  \begin{equation}\label{Phi-Painleve III-zero}
  \Phi(\lambda)=\Phi(0)\left (I+\sum_{k=1}^{\infty}\Phi_k \lambda^k\right )\lambda^{\frac 12n\sigma_3}e^{-\frac {s i}{2\lambda}\sigma_3}\quad \mbox{as}\quad\lambda\rightarrow 0 ,
  \end{equation}
where
  \begin{equation}\label{Phi-Painleve III-zero-coeff}
  \Phi(0)= \left(\frac {n i}s\right)^{(\frac{n}{2}+k) \sigma_3}\left(
                                                                                                                 \begin{array}{cc}
                                                                                                                  1 & c_{k,n}(t) \\
                                                                                                                 - q_{k,n}(t) &1-c_{k,n}(t)q_{k,n}(t)\\
                                                                                                                 \end{array}
                                                                                                               \right)
                                                                                                               \pi_k(0)^{\sigma_3}
   \left(\frac{n i}{s} \right)^{-\frac 12n\sigma_3},
   \end{equation}
   with
      \begin{equation*}
     c_{k,n}(t)=\frac {\pi_k(0)}{2\pi i} \int_{\Gamma}\frac {\pi_k(z)w(z; t)}{z}dz~~\mbox{and}~~q_{k,n}(t)=2\pi i \gamma^2_{k-1,n}(t)\frac {\pi_{k-1}(0)}{\pi_k{(0)}}. \end{equation*}
\end{itemize}

Now, from the above RH problem, we derive the following Lax pair for the function $\Phi(\l, s)$, which is exactly the same as the Lax pair for Painlev\'{e} III;  see \cite[pp.195-203]{fikn}.

\begin{prop} \label{prop-phi-piii}
  For the matrix function $\Phi(\l, s)$ given in \eqref{Psi function-painleve III}, we have
  \begin{equation}\label{Lax-pair-Phi}
    \Phi_{\l}= A(\l,s) \Phi \quad \textrm{and} \quad  \Phi_s = B(\l,s) \Phi,
  \end{equation}
  where
  \begin{equation}\label{Phi-Painleve III-Lax pair-in z}
        A(\lambda,s)=\frac {is}{2}\sigma_3+\frac {A_{-1}}{\lambda}+\frac {A_{-2}}{\lambda^2} \quad \textrm{and} \quad
    B(\lambda,s)=\frac {i\lambda}{2}\sigma_3+B_0+\frac {B_{-1}}{\lambda}.
 \end{equation}
 Here, the coefficients in the above formula are given below
\begin{equation}\label{Phi-Painleve III-Lax pair-in z-coeff-A-1}
A_{-1}=\left(
                                                           \begin{array}{cc}
                                                           \frac n2 + k & -\frac {n}{2\pi i\gamma_{k,n}^2} \left(\frac {ni}s\right)^{n+2k} \\
                                                           2 n \pi i \gamma_{k-1,n}^2 \left(\frac {ni}s\right)^{-n-2k} & -\frac n2-k\\
                                                           \end{array}
                                                         \right)
\end{equation}
\begin{equation}\label{Phi-Painleve III-Lax pair-in z-coeff-A-2}
A_{-2}=\frac {is}{2}\left(
                                                               \begin{array}{cc}
                                                                 1-2c_{k,n}q_{k,n} & -2c_{k,n} \left(\frac {ni} s\right)^{n+2k}\\
                                                                 -2q_{k,n}(1-c_{k,n}q_{k,n})\left(\frac {ni} s \right)^{-n-2k} & 2c_{k,n}q_{k,n}-1\\
                                                               \end{array}
                                                             \right)
\end{equation}
and
\begin{equation}\label{Phi-Painleve III-Lax pair-in z-coeff-B}
B_{0}=\frac {1}{s}\biggl(A_{-1}-\left ( \frac  n 2+k\right )\sigma_3\biggr), \quad B_{-1}=-\frac {A_{-2}}{s}.
\end{equation}

\end{prop}

\begin{proof}
  Note that the jump matrices in \eqref{Phi-Painleve iii-jump} are independent of $\lambda$ and $s$. This implies that both
\begin{equation}\label{Lax-pair-painleve III}
 A(\lambda,s)=\Phi_{\lambda}\Phi^{-1}
 \quad \textrm{and} \quad B(\lambda,s)=\Phi_{s}\Phi^{-1}
\end{equation}
are analytic  functions of $\lambda$  with only possible isolated singularities at the origin and at infinity.   Using the asymptotic expansions in \eqref{Phi-Painleve III-infty}-\eqref{Phi-Painleve III-zero-coeff}, we find that
\begin{equation}
  A_{-1}=\left (\frac n2+k\right )\sigma_3+\frac {is}2[\Phi_{-1},\sigma_3], \qquad A_{-2}=\frac {is}{2}\Phi(0)\sigma_3\Phi(0)^{-1}
\end{equation}
and
\begin{equation}
  B_{0}=\frac i 2[\Phi_{-1},\sigma_3], \qquad  B_{-1}=-\frac {A_{-2}}{s},
\end{equation}where $[X, Y]=XY-YX$ is the commutator.
Then direct computations give us the results.
\end{proof}

It is known  in several circumstances that  the Hankel determinants admit an interpretation as the Jimbo-Miwa-Ueno isomonodromic $\tau$-function for the rank 2 linear system  of differential equations; see \cite{fik} for
the Hankel determinants associated with
the  exponential weight and  \cite{beh,beh2} for more general semi-classical weights. Now we have established  the relation  between the perturbed Laguerre orthogonal polynomials and the Lax pair for the Painlev\'{e} III equation. Naturally, the associated Hankel determinant  is also expected to relate  to the $\tau$-function of the Painlev\'{e} III equation. Thus  we are in a position   to prove our first result for   fixed degree $k$.

\bigskip

\noindent\emph{Proof of Theorem \ref{Theorem: Hankel as tau function}.} According to Proposition \ref{prop-phi-piii}, $\Phi(\l, s)$ satisfies the same Lax pair as Painlev\'{e} III. Then, applying an argument in \cite[(5.3.4),(5.3.7)]{fikn}, we see that the function
\begin{equation}\label{painleve iii-quantity}
u(s)=-i(A_{-1})_{12}/{(A_{-2})_{12}} =-\frac n{2\pi i s c_{k,n}\gamma_{k,n}^2}
\end{equation}
 solves  the Painlev\'{e} III equation
\begin{equation}\label{painleve III-y}
u''(s)=\frac {(u')^2}{u}-\frac {u'}{s}+\frac {4}{s}(\Theta_0 u^2+1-\Theta_{\infty})+4u^3-\frac {4}{u},
\end{equation}
with the parameters $\Theta_0=n$ and $\Theta_{\infty}=-(2k+n)$. By \eqref{a-n in terms of  Y}, we have
\begin{equation}\label{u-and-a}
  u(s)=-\frac {n t}{s a_{k,n}}=\frac {\sqrt{-t}}{i a_{k,n}}.
\end{equation}
Substituting \eqref{u-and-a} into  \eqref{painleve III-y}  gives  us \eqref{introduction: painleve III}.

Next, we consider the Hankel determinant $D_k[w;t]$. Denote by $\Phi_{\infty}(\lambda)$ and $\Phi_{0}(\lambda)$ the series in the expansions \eqref{Phi-Painleve III-infty} and \eqref{Phi-Painleve III-zero}, namely,
\begin{equation*}
\Phi_{\infty}(\lambda)=I+\sum_{k=1}^{\infty}\frac {\Phi_{-k}}{\lambda^k} \quad \textrm{and} \quad \Phi_{0}(\lambda)=I+\sum_{k=1}^{\infty}\Phi_{k}\lambda^k,
\end{equation*}
with $\Phi_{-1}$ qiven in \eqref{Phi-Painleve III-infty-coeff} and
\begin{equation}\label{phi-painleve III zero coeff}
\Phi_1=\frac {1}{is}\left(-n^2\beta_{k,n}(t)-(k^2+nk)+\frac {s^2}{2}(1-2c_{k,n}(t)q_{k,n}(t))\right)\sigma_3+\left(
         \begin{array}{cc}
          0  & * \\
           * & 0 \\
         \end{array}
       \right);
\end{equation} cf. \eqref{Y-solution} and \eqref{Psi function-painleve III},
where $*$ denotes the off-diagonal  entries independent of $\lambda$. By the general theory of Jimbo-Miwa-Ueno \cite{jmu}, the isomonodromy $\tau$-function for the  Lax pair  in \eqref{Lax-pair-Phi}-\eqref{Phi-Painleve III-Lax pair-in z-coeff-B} is defined by  the formula
\begin{equation}\label{painleve III-tao-residue formula}
d\log\tau(s):=-\Res_{\lambda=0}\mathrm{Tr}\left\{\Phi_{0}^{-1}(\lambda)\frac {\partial \Phi_0(\lambda)}{\partial \lambda}dT_0(\lambda)\right \}-\Res_{\lambda=\infty}\mathrm{Tr}\left\{\Phi_{\infty}^{-1}(\lambda)\frac {\partial \Phi_{\infty}(\lambda)}{\partial \lambda}dT_{\infty}(\lambda)\right\},
\end{equation}
where
\begin{equation*}
dT_0(\lambda)=-\frac {i}{2\lambda}\sigma_3ds, \quad dT_{\infty}(\lambda)=\frac {i\lambda}{2}\sigma_3ds;
\end{equation*}
see \cite[Eq.(1.23)]{jmu}. Substituting the definition of $\Phi_{\infty}(\lambda)$ and $\Phi_{0}(\lambda)$ into \eqref{painleve III-tao-residue formula}, we obtain
\begin{equation}\label{painleve III-tao-residue formula-1}
\frac {d\log\tau(s)}{ds}=\frac i2 \mathrm{Tr}(\Phi_{1}\sigma_3-\Phi_{-1}\sigma_3) =\frac 1s\left (-2n^2 \beta_{k,n}+2s^2c_{k,n}q_{k,n}-s^2+(k^2+nk)\right ).
\end{equation}
Now a combination of \eqref{differential identity}, \eqref{H and a-n}, \eqref{p and alpha-n} and \eqref{beta integrate by part} gives
\begin{equation}\label{painleve III-first integral}
n^2\beta_{k,n}=n^2t c_{k,n}q_{k,n}-H_{k,n}+k(k+n);
\end{equation}see \eqref{Phi-Painleve III-Lax pair-in z-coeff-A-2} for the definition of $c_{k,n}(t)$ and $q_{k,n}(t)$.
Thus, we obtain  from \eqref{painleve III-tao-residue formula-1} and \eqref{painleve III-first integral} that
\begin{equation}\label{painleve III-tao-residue formula-2}
\frac {d\log\tau(s)}{dt} =\frac 1{2t}\left (2H_{k, n}-s^2-(k^2+nk)\right ).
\end{equation}
Here use has been made of  the relation $s^2=n^2t$.   In view  of  the formula \eqref{logrithmic derivative of hankel}, and integrating both sides of \eqref{painleve III-tao-residue formula-2}, we arrive at the following relation between the Hankel determinant $D_k[w;t]$ and the $\tau$-function of the Painlev\'{e} III equation:
\begin{equation*}
D_k[w;t]=\mathrm{const} \cdot \tau(s) e^{  {n^2t}/{2}}t^{ {k(k+n)}/{2}},
\end{equation*}which is \eqref{introduction:painleve III-tao function}.
This completes the proof of Theorem \ref{Theorem: Hankel as tau function}. \hfill \qed


\section{Equilibrium measures} \label{sec-E-measure}

The equilibrium measure with the external field $V_t(x)$ in \eqref{potential} is given recently in Texier and Majumdar \cite{texier}. To obtain  a double scaling limit at the critical time, we need a modified equilibrium problem, which will involve  a \emph{signed} measure.  This signed measure will be used to construct the important $g$-function and $\phi$-function in the Riemann-Hilbert analysis. The idea of considering a modified equilibrium problem has been successfully applied to study similar double scaling limits in several different problems, such as   varying quartic potentials by Claeys and Kuijlaars \cite{claeys:Kuijlaars} and Duits and Kuijlaars \cite{Dui:Kuij}, and a cubic potential by Bleher and Dea\~{n}o \cite{Ble:Dea}.

In this section, we will go back to the weight  \eqref{pL weight},   consider a regular equilibrium problem first, and see how the critical time occurs.  Then, to facilitate our future Riemann-Hilbert analysis near the critical time, we will consider a modified equilibrium problem by fixing the left endpoint. This will give us the signed measure we need.

\subsection{Equilibrium measure  and a critical case} \label{sec-E-measure-regular}

Consider the extremal problem minimizing the energy with the external field $V_t(x)$ in \eqref{potential}:
\begin{equation}\label{energy}
I( \nu )=\int_0^{\infty}V_t(x)d\nu(x)+\int_0^{\infty}\int_0^{\infty}\log\frac1{|x-y|}d\nu(x)d\nu(y).
\end{equation}
According to the general  potential theory \cite{Saff:Totik}, there exists a unique minimizer $d \nu_t$ of $I(\nu)$ {  among all Borel probability measures $d\nu$ on $[0, +\infty)$}, such a   probability measure   is called the equilibrium measure. For the potential $V_t(x)=x-\log x+t/x$ in \eqref{potential}, the equilibrium measure $d \nu_t$ can be computed explicitly.

{
The equilibrium  measures   for $t>t_{cr}$ and   $t\to t_{cr}$  have been  computed explicitly in Texier and Majumdar \cite{texier}. To make the present  paper self-contained, we sketch the proof below, which  differs from that in \cite{texier}. Inspired by \cite{texier}, and in view of the measure for the positive-$t$ case,   we    assume  that the support of $d\nu_t(x)$ has only one piece.   Also,  for fixed $t$, the behavior of the  density
$v_t(x)$ is expected to demonstrate a weak singularity at the endpoints since the contour is deformed to keep away from the possible singularity at the origin. We derive the equilibrium measure by solving a scalar RH problem,   based on the Euler-Lagrange equation \eqref{Lagrange variation equation}.}
\begin{prop}
  Let $v_t(x)$ be the density function of the equilibrium measure $d\nu_t$ supported on an  interval $(a, b)\subset [0,  \infty)$, such that $d\nu_t(x)=v_t(x) dx$. Then {  for $t>t_{cr}$} we have
  \begin{equation} \label{vt-formula}
    v_t(x)=\frac {x+c}{2\pi x^2}\sqrt{(x-a)(b-x)}, \qquad x\in (a,b),
  \end{equation}
  where $c=t/\sqrt{ab}$ and  $a,b$ are determined by \eqref{a,b equation}.
  Moreover, when $t= t_{cr}$ as   in \eqref{t-critical},
    we have
  \begin{equation}\label{psi-critical}
   v_{cr}(x)=\frac {1}{2\pi x^2}\sqrt{(x-a_{cr})^3(b_{cr}-x)}, \qquad x\in (a_{cr},b_{cr}),
  \end{equation}
  where the critical endpoints $a_{cr}$ and  $b_{cr}$ are given in \eqref{a,b, mu critical}.
\end{prop}

\begin{proof}
 From \eqref{energy}, it is known that the equilibrium measure $d\nu_t$ satisfies the  Euler-Lagrange equation
\begin{equation}\label{Lagrange variation equation}
V_t(x)+2\int_a^{b}\log\frac1{|x-y|}d\nu_t(y)=l, \qquad x\in(a,b),
\end{equation}
where $l$ is the Lagrange multiplier.
Differentiating with respect to $x$, we get
\begin{equation}\label{singular integral}
V'_t(x)-2 \, \textrm{p.v.}\int_a^{b}\frac{v_t(y)}{x-y}dy=0,\quad  x\in(a,b).
\end{equation}
where the integral is taken as the Cauchy principle value. This is an integral equation for the density function $v_t(x)$, which can be solved explicitly. Indeed, one can define
\begin{equation}\label{G function}
G(z)=\frac 1{\pi i}\int_a^b \frac {v_t(y)}{y-z}dy, \quad z \in \mathbb{C} \setminus [a,b].
\end{equation}
Then it follows from the Plemelj formula that
\begin{equation}\label{G function boundary value}
G_{\pm}(x)=\frac 1{\pi i}\textrm{p.v.} \int_a^b \frac {v_t(y)}{y-x}dy \pm v_t(x), \quad x\in(a,b),
\end{equation}
where the integral is  the Cauchy principle value. It is readily verified that   $G(z)$ satisfies  the following  scalar Riemann-Hilbert problem:
\begin{itemize}
  \item[(i)] $G(z)$ is analytic for $z \in \mathbb{C}\setminus[a,b]${ , having  at most weak singularities at $z=a,\;b$};

  \item[(ii)] $ G_{+}(x)+G_-(x)=-\frac{1}{\pi i}V'_t(x)$ for  $x\in(a,b) $;

  \item[(iii)] $ G(z)\sim-\frac 1{\pi iz}$ as $z\to \infty$.

\end{itemize}
Solving this RH problem yields
\begin{equation}\label{G function-formula}
G(z)=-\frac {V_t'(z)}{2\pi i}+ \frac{\tilde c z+c}{2\pi i z^2} \sqrt{(z-a)(z-b)},
\end{equation}where $V_t'(z)=1-\frac 1 z-\frac t{z^2}$, $c=\frac {t}{\sqrt{ab}}$  and $\tilde c=\frac 1 {\sqrt{ab}} \left [ 1+\frac t 2\left (\frac 1 a+\frac 1 b\right )\right ]$. Here attention should be paid to the fact that $G(z)$ is analytic outside the interval $[a, b]$, especially at $z=0$.
Now expanding \eqref{G function-formula} in powers of $1/z$,
the large-$z$ behavior of $G(z)$ ensures that
\begin{equation*}
  \tilde c=1~~\mbox{and}~~ c-\frac {a+b} 2=-3,
\end{equation*}which are indeed
\eqref{a,b equation}. Furthermore, a combination of \eqref{G function boundary value} and \eqref{G function-formula} yields \eqref{vt-formula}.

Moreover, in the critical case when
\begin{equation*}
c_{cr}=-a_{cr}
\end{equation*}
a straightforward computation gives us \eqref{t-critical}-\eqref{a,b, mu critical}.
\end{proof}

\begin{rmk}
  Of course, the formulas for the density function and endpoints in \eqref{vt-formula} and \eqref{a,b equation}  hold when $t\geq0$. For any $t\geq 0$, the density function is supported on $[a,b]$ with $0<a<b$ and vanishes like square roots at both endpoints. As a consequence, one will obtain usual Airy-type and sine-type asymptotic expansions for the orthogonal polynomials near the endpoints and inside the support, respectively.
\end{rmk}


\subsection{Signed   equilibrium   measure}

The critical case when $t=t_{cr}$ is termed a \emph{freezing transition}; see \cite{texier}.
One can see that, when $t=t_{cr}$, the density function $v_t(x)$ in \eqref{psi-critical} vanishes like a $3/2$ root at $a_{cr}$. This suggests that the local behavior of the orthogonal polynomials near $a_{cr}$ is described in terms of the Painlev\'e I transcendents; see \cite{Ble:Dea,Dui:Kuij}. To precisely construct a local parametrix near the endpoint $a_{cr}$ by using the Painlev\'e I transcendents, a delicate study near $a_{cr}$ is needed in our subsequent nonlinear steepest descent analysis for the RH problem. Therefore, technically it is more convenient to have a measure whose left endpoint of the support is exactly located at $a_{cr}$.  Note that in the case when only positive measures are involved as in Section \ref{sec-E-measure-regular}, both endpoints $a$ and $b$ in \eqref{a,b equation} vary when the value of parameter $t$ changes.   So we need to minimize the same energy functional  \eqref{energy}
{  among  \emph{signed} measures   on $[a_{cr}, + \infty)$  which are nonnegative
except possibly on $[a_{cr}, a_{cr} + \delta_1 ]$ for some sufficiently small $ \delta_1> 0$.
}  As there is no symmetry as in \cite{Dui:Kuij}, the right endpoint $b$ may depend  on $t$. A similar treatment is also employed in \cite{Ble:Dea}.

By a similar argument performed in Section \ref{sec-E-measure-regular}, we find the new minimizer explicitly.

\begin{prop}
  Let $\psi_{t}(x)$ be the signed  density function of the minimizer
  of the minimizer of the energy functional in \eqref{energy} {  on $[a_{cr}, + \infty)$}.
   Then we have
  \begin{equation}\label{E-measure-modified}
    \psi_{t}(x)=\frac {1}{2\pi x^2}\sqrt{\frac{b-x}{x-a_{cr}}} \left(x^2+d_1x+d_0\right),~~x\in [a_{cr}, b],
  \end{equation}
  where
  \begin{equation}
    d_0=-t\sqrt{\frac{a_{cr}}{b}} \quad \textrm{and} \quad d_1=-\sqrt{\frac{a_{cr}}{b}}\left(1-\frac{t}{2a_{cr}}+\frac{t}{2b}\right)
  \end{equation}
  and  $b$ is determined by the equation
  \begin{equation}
        \sqrt{\frac{a_{cr}}{b}}\left(1-\frac{t}{2a_{cr}}+\frac{t}{2b}\right)+\frac{b-a_{cr}}{2}=3.
  \end{equation}
\end{prop}
\vskip .5cm

It is worth noting  that for $t=t_{cr}$, the density of the modified equilibrium measure $\psi_{t}(x)$ is reduced to $v_{cr}(x)$ in \eqref{psi-critical}. Moreover, near the critical time $t=t_{cr}$, we have
\begin{equation}\label{e-measure-modified coefficients}
 \begin{array}{ll}
           & b=b_{cr}+\sqrt{\frac{b_{cr}}{a_{cr}}}(b_{cr}-a_{cr})^{-1}(t-t_{cr})+O\left ((t-t_{cr})^2\right ), \\[.2cm]
           & d_0=a_{cr}^2-\sqrt{\frac{a_{cr}}{b_{cr}}} \frac{2b_{cr} - a_{cr}}{2(b_{cr} - a_{cr})}(t-t_{cr})+O\left ((t-t_{cr})^2\right ),\\[.2cm]
           & d_1=-2a_{cr}+\frac 12\sqrt{\frac{b_{cr}}{a_{cr}}}(b_{cr}-a_{cr})^{-1}(t-t_{cr})+O\left ((t-t_{cr})^2\right ).
         \end{array}
 \end{equation}




Based on the signed measure obtained above, we  define several auxiliary functions which will  be used in our further analysis.
\begin{defn}The  $g$-function is defined as
   \begin{equation}\label{g function}
        g(z)=\int_{a_{cr}}^b\log(z-s)\psi_{t}(s)ds \quad  \mbox{for}\quad z\in \mathbb{C} \setminus (-\infty,b],
    \end{equation}
    where $\arg(z-s)\in(-\pi,\pi)$, and  the equilibrium  density  function $\psi_{t}(x)$ is given in \eqref{E-measure-modified}.
\end{defn}

\begin{defn}
  We also define the following $\phi$-functions
\begin{eqnarray}
\phi_{t}(z)&=&\frac 1{2}\int_{a_{cr}}^z\frac{(s^2+d_1s+d_0)(s-b)^{1/2}}{s^2 (s-a_{cr})^{1/2} }ds, \ z\in\mathbb{ C}\setminus\{(-\infty,0)\cup(a_{cr},\infty)\}, \label{phi function-sign measure} \\
\phi_{cr}(z)&=&\frac 1{2}\int_{a_{cr}}^z\frac{(s-a_{cr})^{3/2}(s-b_{cr})^{1/2}}{s^2}ds, \ z\in\mathbb{ C}\setminus\{(-\infty,0)\cup(a_{cr},\infty)\}, \label{phi function-critical}
\end{eqnarray}
where the branches are chosen such that $\arg(s-a_{cr})\in(0,2\pi)$, $\arg(s-b_{cr})\in(0,2\pi)$ and $\arg(s-b)\in(0,2\pi)$.
\end{defn}

From the above definitions, it is immediately seen that $g(z)$ satisfies the Euler-Lagrange equation
\begin{equation}\label{phase condition}
g_+(x)+g_-(x)-V_t(x)-l=0, \quad x\in(a_{cr},b),
\end{equation}
{  and the variational inequality
\begin{equation}\label{variational-inequality}
2g(x) -V_t(x)-l<0,~x\in  (b, \infty),
\end{equation}where $l$ is the Lagrange multiplier introduced in \eqref{Lagrange variation equation}.
}
Moreover, the $g$-function and the $\phi$-function are related by
\begin{equation}\label{g and phi}
g_+(x)-g_-(x)=2\pi i-2\left (\phi_ t\right )_+(x), \quad x\in(a_{cr},b) .
\end{equation}

Note that $\phi_{t}(z)$ and $\phi_{cr}(z)$ are close to each other when $t$ approaches  $t_{cr}$. If we rewrite $\phi_{t}(z)$ as
\begin{equation}\label{phi-t-phi-critical}
\phi_{t}(z)=\phi_{cr}(z)+(t-t_{cr})\phi_0(z),
\end{equation}
then, in view of \eqref{phi function-sign measure}-\eqref{phi function-critical},  we have
\begin{equation}\label{phi function -zero}
\phi_0(z)=-i\frac{\sqrt{b_{cr}-a_{cr}}}{2a_{cr}\sqrt{a_{cr}b_{cr}}}\sqrt{z-a_{cr}}(1+r_0(z)), \quad  \arg(z-a_{cr})\in(0,2\pi),
\end{equation}
where $r_0(z)$ is analytic in a neighborhood of $z=a_{cr}$ and $r_0(a_{cr})=0$. We also need some local information of the functions $\phi_{cr}(z)$ and $\phi_{t}(z)$ at critical    points   $z=a_{cr}$ and $z=0$. From their definitions in \eqref{phi function-sign measure} and \eqref{phi function-critical}, we have
\begin{equation}\label{phi-critical function at a}
\phi_{cr}(z)\sim \frac{(b_{cr}-a_{cr})^{\frac 12}}{5a_{cr}^2}(z-a_{cr})^{\frac 52}e^{\frac 12\pi i}, \quad z\to a_{cr},
\end{equation}
where $\arg(z-a_{cr})\in(0,2\pi)$, and
\begin{equation}\label{phi function at 0}
\phi_{t}(z) = \frac {t}{2z} - \frac{1}{2} \log z + O(1), \quad z\to 0,
\end{equation}
where $\arg z\in(-\pi,\pi)$. From the above formula, one can see that $\re \phi_{t}(z)>0$
if $t<0$  and  $z$ approaches  the origin such that $\cos (\arg z)< \frac {|z|\log |z| } t$; see Figures \ref{contour-Phi-two} and \ref{contour-Phi-cr}.
{  In particular, we see that  $e^{-n\phi_t(z)}$ is exponentially small as $z\to 0$, $z\in \Gamma_2$ or $z\in \Gamma_3$; cf. Figure \ref{contour-ortho} for the contours.  In view of \eqref{phi-t-phi-critical}, one can see that the same holds for $e^{-n\phi_{cr}(z)}$. It is worth mentioning that on $\Gamma_2$ and $\Gamma_3$ with $|z-\delta|=\delta$, we have $\re \frac 1 z \equiv \frac 1 {2\delta}$. }

\begin{figure}[h]
 \begin{center}
   \includegraphics[width=15.5cm]{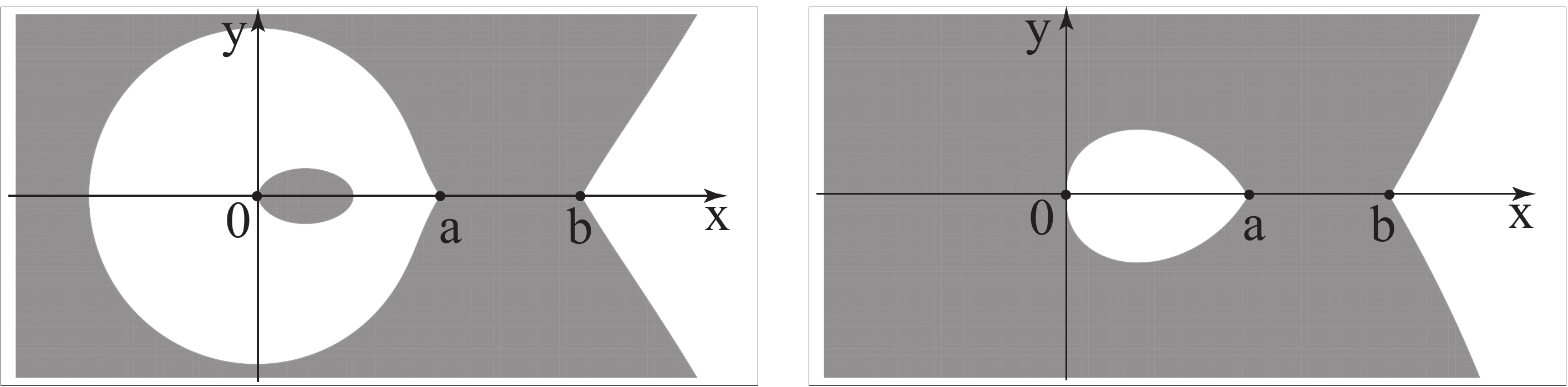} \end{center}
  \caption{The shaded region is the region where $\re\phi_t(z)<0$. The left and right pictures correspond to cases when $t<0$ and $t>0$, respectively.}
 \label{contour-Phi-two}
\end{figure}

\begin{figure}[h]
 \begin{center}
   \includegraphics[width=11cm]{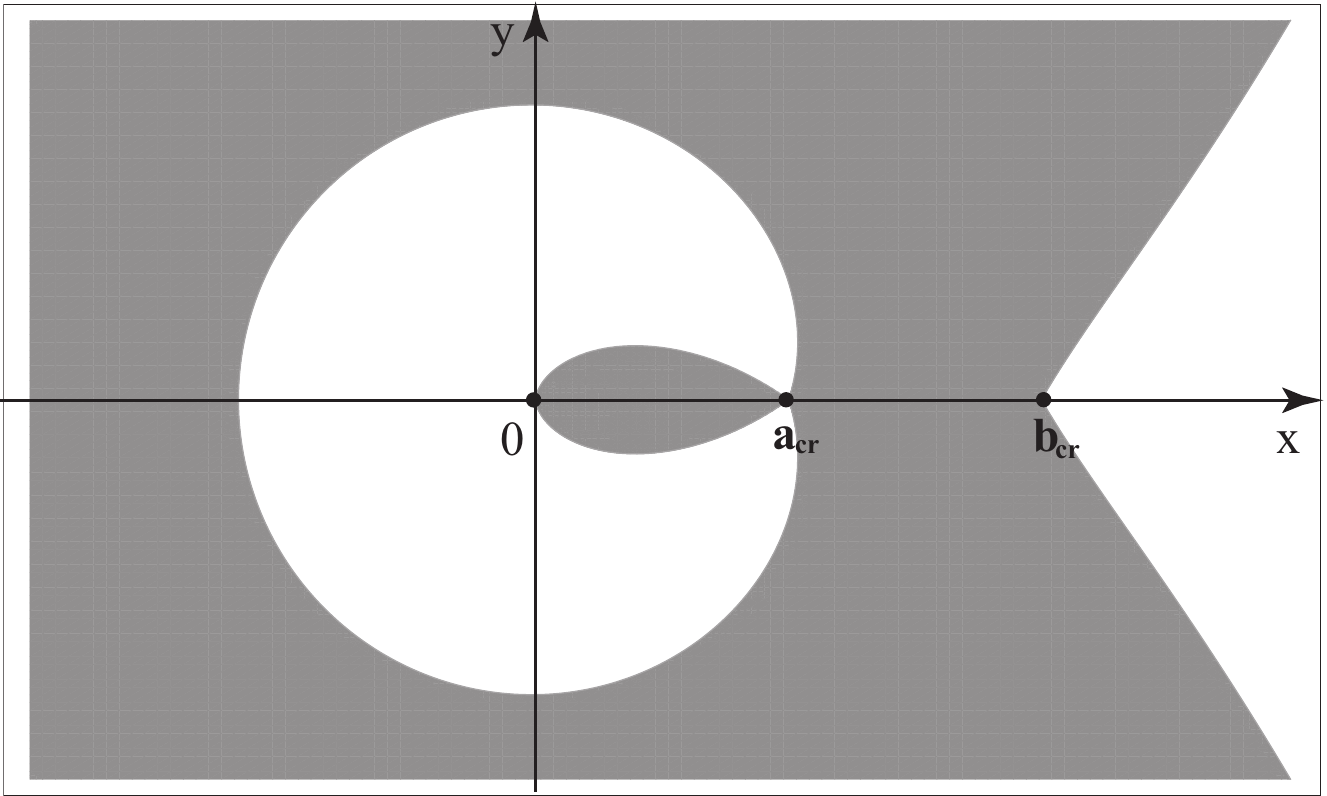} \end{center}
  \caption{The shaded region is the region where $\re\phi_{cr}(z)<0$.  Note that this figure is not the exact one  for $\phi_{cr}$ defined in \eqref{phi function-critical}. Here we have rescaled the figure, especially near $a_{cr}$,  for better illustration: because the exact value of $a_{cr}$ is too small as compared with $b_{cr}$; see \eqref{a,b, mu critical}.     }
 \label{contour-Phi-cr}
\end{figure}


\section{Nonlinear steepest descent analysis} \label{sec-RH-analysis}

In this section, we apply the nonlinear steepest descent method developed by Deift and Zhou et al.\;\cite{dkmv1,dkmv2} to the RH problem for
$Y$. The idea is to obtain, via a series of invertible transformations $ Y \to T \to S \to R$, the RH problem for $R$ whose jump matrices are close to the identity ones.

\subsection{The first transformation $Y\to T$: Normalization at infinity} \label{sec-nomalization}

We make use of the $g$-function defined in \eqref{g function} to normalize the RH problem for $Y$ in Section \ref{sec-rhp-OPs} when $k=n$. As $g(z) \sim \log z$ for large $z $, we introduce the first transformation $Y \to T$ as follows:
\begin{equation}\label{TrsnaformationY-T}
    T(z)=e^{-\frac{n l}{2}\sigma_3}Y(z) e^{-n(g(z)-\frac{l}{2})\sigma_3},
\end{equation}
where $l$ is the Lagrange multiplier in \eqref{phase condition}. Then, $T$ solves the following RH problem.
\begin{description}
\item(T1)~~ $T(z)$ is analytic in $\mathbb{C}\backslash \Gamma$; see Figure \ref{contour-ortho} for $\Gamma=\Gamma_1\cup\Gamma_2\cup\Gamma_3$;

\item(T2)~~  The jump condition is
\begin{equation}\label{T-jump}T_+(z)=T_-(z)
\left(
                               \begin{array}{cc}
                                 e^{n(g_ -(z)-g_+(z))} & c_je^{n(-V_t(z)+g_+(z)+g_-(z)-l)} \\
                                 0 & e^{n(g_ +(z)-g_-(z))} \\
                                 \end{array}
                             \right)
\end{equation}
for $z \in \Gamma_j$, $j=1,2,3$, where $V_t(z)$ is defined in \eqref{potential new},  and $c_1=1$,  $c_2=\alpha$, $c_3=1-\alpha$;

\item(T3)~~   The asymptotic behavior of $T(z)$  at infinity is
\begin{equation}\label{T-infty}T(z)= I+O(1/z)\quad \mbox{as}\quad z\rightarrow \infty .\end{equation}
\end{description}

Appealing   to the properties of $g(z)$ and $\phi_t(z)$ in \eqref{phase condition} and \eqref{g and phi}, the jump matrices in \eqref{T-jump} can be expressed in terms of the function $\phi_t(z)$ as follows:
\begin{equation}\label{T-jump-in-phi}
T_+(z)=T_-(z)\left\{
\begin{array}{ll}
 \left(
                               \begin{array}{cc}
                                 1 & c_je^{-2n \phi_t(z)} \\
                                 0 & 1 \\
                               \end{array}
                             \right),&  z\in \Gamma_j \setminus (a_{cr},b),~j=1,2,3;\\[.4cm]
                                \left(
                               \begin{array}{cc}
                                e^{2n\left (\phi_t\right )_+(z)}& 1 \\
                                 0 &e^{2n\left (\phi_t\right )_-(z)} \\
                                 \end{array}
                             \right), &  z\in(a_{cr},b).
\end{array}
\right .
\end{equation}

\subsection{The second transformation $T\to S$: Contour deformation} \label{sec-S}

Since  $\left (\phi_{t}\right )_{\pm}(z)$ are  purely imaginary on $(a_{cr},b)$, the jump matrix for $T(z)$ on $z\in(a_{cr},b)$ possesses  highly oscillatory diagonal entries. To remove the  oscillation, we open the lens near $(a_{cr},b)$ and
introduce the second transformation:
\begin{equation}\label{transformationT-S}
S(z)=\left \{
\begin{array}{ll}
  T(z), & \mbox{for $z$ outside the lens shaped region;}
  \\  [.4cm]
  T(z) \left( \begin{array}{cc}
                                 1 & 0 \\
                                  -e^{2n\phi_t(z)} & 1 \\
                               \end{array}
                             \right) , & \mbox{for $z$ in the upper lens
                             region;}\\[.4cm]
T(z) \left( \begin{array}{cc}
                                 1 & 0 \\
                                    e^{2n\phi_t(z)} & 1 \\
                               \end{array}
                             \right) , & \mbox{for  $z$  in  the  lower
                             lens
                             region. }
\end{array}\right .
\end{equation}

\begin{figure}[h]
 \begin{center}
   \includegraphics[width=12cm]{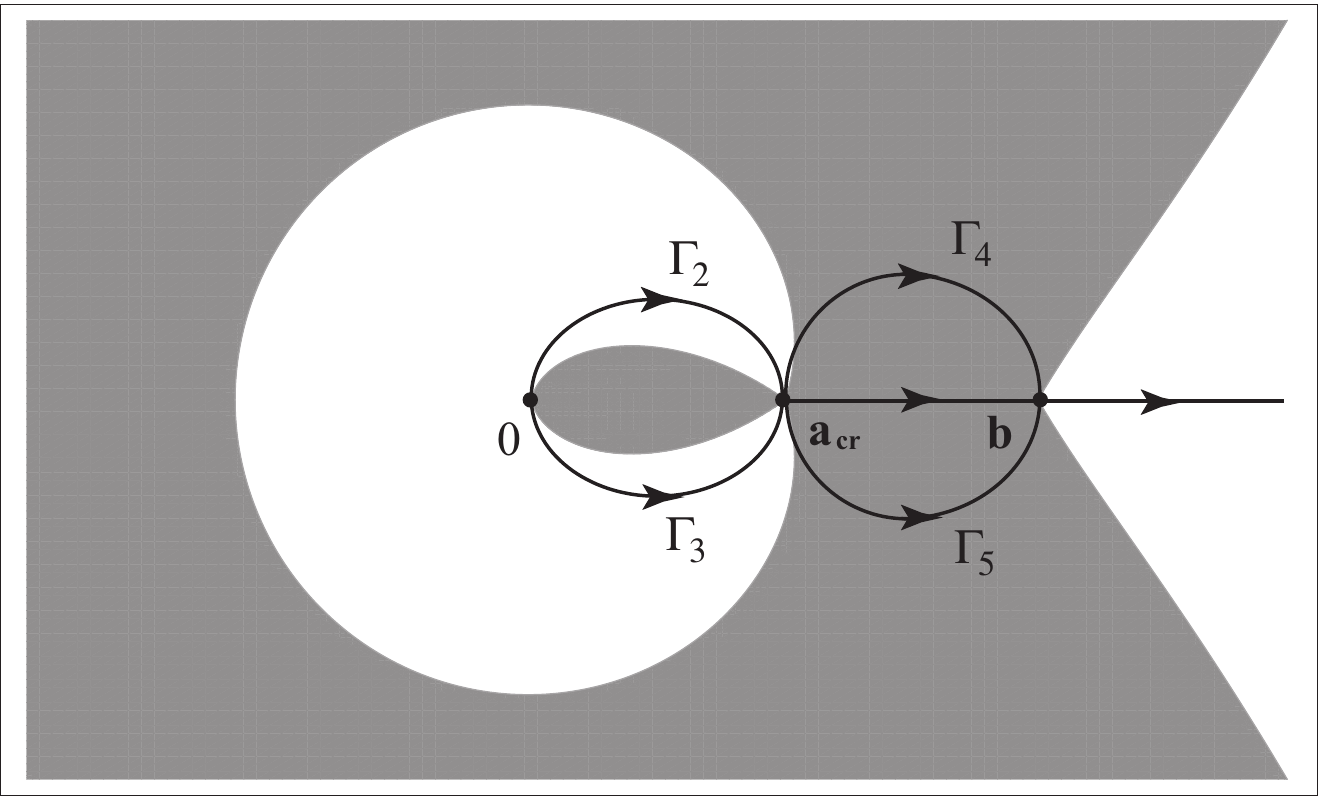} \end{center}
  \caption{Contour $\Sigma_S=(a_{cr}, b)\cup (b, \infty) \cup_{j=2}^5 \Gamma_j$. The shaded region is the region where $\re\phi_{cr}(z)<0$.}
 \label{contour-S}
\end{figure}

Then $S(z)$ solves the RH problem

\begin{description}
  \item(S1)~~ $S(z)$ is analytic  in  $\mathbb{C}\setminus\Sigma_S$; see Figure \ref{contour-S} for the contours;
  \item(S2)~~The jump conditions are
  \begin{equation}\label{S-jump}S_+(z)=S_-(z)  \left\{\begin{array}{ll}
                 \left( \begin{array}{cc}
                                 1 & 0 \\
                                  e^{2n\phi_t(z)} & 1 \\
                               \end{array}
                             \right), &  z \in \Gamma_4\cup \Gamma_5 ,\\ [.4cm]
               \left( \begin{array}{cc}
                                0&  1 \\
                                   -1 & 0 \\
                               \end{array}
                             \right), &  z\in(a_{cr},b), \\ [.4cm]
                  \left( \begin{array}{cc}
                                 1 &  \alpha e^{-2n \phi_t(z)} \\
                                  0  & 1 \\
                               \end{array} \right), &   z\in \Gamma_2,\\[.4cm]
                  \left( \begin{array}{cc}
                                 1 &  (1-\alpha)e^{-2n \phi_t(z)} \\
                                  0  & 1 \\
                               \end{array} \right), &   z\in  \Gamma_3,\\[.4cm]
                  \left( \begin{array}{cc}
                                 1 &  e^{-2n \phi_t(z)} \\
                                  0  & 1 \\
                               \end{array}\right), &   z\in (b,+\infty).
                \end{array}\right.
  \end{equation}
  \item(S3)~~ The asymptotic behavior at infinity is
  \begin{equation}\label{S-infty}
  S(z)= I+O( 1/ z),~~~\mbox{as}~~z\rightarrow \infty .
  \end{equation}
 \end{description}

To study the asymptotic behavior of $S(z)$ for large $n$, we may exam  the signs of $\re \phi_t(z)$, to see if the jumps are of the form $I$ plus exponentially small terms.
Special attention should be paid  in the present case since we are dealing with the signed measure \eqref{E-measure-modified}.  Fortunately, when $n$ is large and $t-t_{cr}$ is small enough, we can still determine the  signs
of $\re \phi_t(z)$ near the endpoint $a_{cr}$. Similar discussions can be found  in \cite{Ble:Dea,Dui:Kuij} where modified equilibrium problems are also addressed.

\begin{prop} \label{prop-phi-sign}
   Let $U$ be a neighbourhood of $a_{cr}$. Then, for any $\varepsilon > 0$, there exists a {  $\delta_T>0$} such that for all $t \in \mathbb{R}$ with $|t-t_{cr}|<{  \delta_T}$, we have $\re \phi_t(z) < -\varepsilon $ on the upper and lower lips of the lens on the outside of $U$, namely, $z\in \{\Gamma_4\cup\Gamma_5\}\setminus U$. Moreover,{
   there exists a positive $r$, such that
    $\re \phi_t(z)> \varepsilon $ on  $\{\Gamma_2\cup\Gamma_3\}\setminus U$ and on $[b+r, \infty)$.  }
\end{prop}
\begin{proof}
{   In view of \eqref{e-measure-modified coefficients}, we see that the factor $x^2+d_1x +d_0$ in \eqref{E-measure-modified} possesses a pair of zeros
\begin{equation*}
  x_\pm =a_{cr} \pm \frac 1 {\sqrt 2} \left ( \frac { a_{cr}} {b_{cr}}\right )^{1/4} (t-t_{cr})^{1/2} +O(t-t_{cr}).
\end{equation*} One can choose $\delta_T$ small enough,  so that $x_\pm\in U$.
Similar to }
  those {  conducted} in \cite[Prop.\;4.2]{Dui:Kuij} and \cite[p.23]{Ble:Dea}, {  we can prove that  the  jumps on  the portions  of $\Gamma_4$ and $\Gamma_5$, outside of $U$ and keeping a distance from the soft edge $b$,
  are of the form $I$ plus an exponentially small term.

  Next,  we   estimate $\re \phi_t(z)$ on $\Gamma_2\cup\Gamma_3\cup (b, \infty)$ in a straightforward manner, using the  explicit representations of the $\phi$-functions  \eqref{phi function-sign measure}, \eqref{phi function-critical}  and \eqref{phi-t-phi-critical}.    In view of  \eqref{phi-t-phi-critical}, we need only check the critical case for $\phi_{cr}$, and it is readily seen from \eqref{phi function-critical} that
   \begin{equation*}
    \re \left (\phi_{cr}\right )_\pm(x)=  \frac 1{2}\int_{b_{cr}}^x\frac{(s-a_{cr})^{3/2}(s-b_{cr})^{1/2}}{s^2}ds>0~~\mbox{for}~ x>b_{cr}.
   \end{equation*}
 For $z$ on the semi-circle $\Gamma_2$; cf. Figure \ref{contour-ortho}, we take the integration path to be the arc from $a_{cr}$ to $z$, and   adapt the parametrization $s=\delta+\delta e^{i\theta}$, where $\delta=a_{cr}/2$ and $\theta \in [0, \pi)$.
  As a result, we have
    \begin{equation}\label{phi function-critical-real}
  \phi_{cr}(z)= \frac {\sqrt {a_{cr}}} 4 \int_{0}^{\theta_z}\left\{ |s-b_{cr}|^{\frac 12} \left (\sin \frac \theta 2\right ) ^{\frac 3 2} \left (\cos\frac \theta 2\right )^{-2}\right\}   e^{i\left ( \frac {5\pi} 4+\frac {3\theta} 4+\frac 1 2 \arg (s-b_{cr})\right )} d\theta,
  \end{equation}
 where $\theta_z\in [0, \pi)$ such that $z=\delta+\delta e^{i\theta_z}$.  What is more, we have $\arg (s-b_{cr})\in [ \pi-\arcsin A, \pi]\subset [3\pi/4,    \pi]$ on $\Gamma_2$, where $A=\frac {a_{cr}}{2b_{cr}-a_{cr}}$ so that  $\arcsin A\approx 0.0021\pi$. Hence the argument of the integrand  lies in the interval $[2\pi -3\pi/8, 2\pi+\pi/2)$, so long as $\theta\in [0, \pi)$.
  Therefore, from \eqref{phi function-critical-real} we see that $\phi_{cr}(a_{cr})=0$,  and $\re \phi_{cr}(z)$ is strictly  monotonically  increasing as $z\in \Gamma_2$ goes away from $a_{cr}$.
 The same result holds for $z\in \Gamma_3$. It is worth mentioning that
    $\re \phi_t(z)\to +\infty$ as  $z\to 0$, $z\in \Gamma_2\cup\Gamma_3$;  see the formula \eqref{phi function at 0} and the discussion that follows. We  note that the estimates on the lens boundaries $\Gamma_4$ and $\Gamma_5$ can also be obtained from the representations  \eqref{phi function-sign measure}, \eqref{phi function-critical}  and \eqref{phi-t-phi-critical}.
     }
\end{proof}

\subsection{The global parametrix} \label{sec-N}

Having had  Proposition \ref{prop-phi-sign}, we see from \eqref{S-jump} that the jump matrix for $S$ is the identity matrix, plus an exponentially
small term, for fixed $z$ bounded away from the interval $(a_{cr},b)$. Neglecting the exponentially small terms, we arrive at an approximating RH problem for $N(z)$ as follows:

\begin{description}
\item(N1)~~  $N(z)$ is analytic  in  $\mathbb{C}\backslash
[a_{cr},b]$;
\item(N2)~~   \begin{equation}\label{N-jump} N_{+}(x)=N_{-}(x)\left(
       \begin{array}{cc}
       0 & 1 \\
       -1 & 0 \\
       \end{array}
       \right)~~~\mbox{for}~~x\in  (a_{cr},b);\end{equation}
\item(N3)~~    \begin{equation}\label{N-infty} N(z)= I+O( 1/ z) ~~~\mbox{as}~~z\rightarrow\infty .\end{equation}
  \end{description}

The solution to the above RH problem is constructed explicitly as
 \begin{equation}\label{N-solution}
 N(z)=M^{-1}    \varrho(z)^{-\sigma_3}   M,
 \end{equation}
 where $M=(I+i\sigma_1) /{\sqrt{2}}$ and $\varrho
(z)=\left(\frac{z-b} {z-a_{cr}}\right)^{1/4}$ with $\arg (z-a_{cr})\in (-\pi, \pi)$ and $\arg (z-b)\in (-\pi, \pi)$.

The jump matrices of $S(z)N(z)^{-1}$ are not uniformly close to the identical  matrix $I$ near the endpoints $a_{cr}$ and $b$, thus local parametrices have to be constructed in neighborhoods of these endpoints.

\subsection{The local  parametrix at the soft edge} \label{sec-local parametrix-Airy}

The local parametrix at the right endpoint $z = b$ is the same as that of the Laguerre polynomials at the soft edge. More precisely, the parametrix is to be constructed in $U(b, r) = \{ z ~| \; |z-b|  <r \}$, $r$ being a fixed positive number, such that

\begin{description}
  \item(a)~~ $P^{(b)}(z)$ is analytic in $U(b,r) \backslash \Sigma_{S}$; see Figure \ref{contour-S} for the contour $\Sigma_{S}$;
  \item(b)~~ In  $U(b,r)$, $P^{(b)}(z)$ satisfies the same jump conditions as $S(z)$ does; cf. \eqref{S-jump};
  \item(c)~~  $P^{(b)}(z)$ fulfils the following  matching condition
  on $\partial U(b,r)$:
\begin{equation}\label{matchingPb-N}
P^{(b)}(z)N^{-1}(z)=I+ O\left ( \frac 1n\right ).
 \end{equation}
\end{description}

The parametrix can be constructed, out   of the Airy function
and its derivative, as in \cite[(3.74)]{vanlessen}; see also \cite{deift,dkmv2}.

\subsection{Local parametrix at the critical point $z=a_{cr}$ and Painlev\'{e} I} \label{sec-local parametrix-painleve}

Now we focus on the construction of the parametrix at the endpoint $a_{cr}$. We seek a parametrix in $U(a_{cr}, r)=\{z~|\;|z-a_{cr}|<r \}$, $r$ being a fixed positive number, such that the following RH problem is satisfied:
\begin{description}
  \item(a)~~ $P^{(a)}(z)$ is analytic in $U(a_{cr},r) \backslash  \Sigma_{S}$; cf. Figure \ref{contour-S};
  \item(b)~~ In  $U(a_{cr}, r)$, $P^{(a)}(z)$ satisfies the same jump conditions as $S(z)$ does; cf. \eqref{S-jump};
  \item(c)~~  $P^{(a)}(z)$ fulfils the following  matching condition
   on  $\partial U(a_{cr},r)$:
\begin{equation}\label{matchingPa-N}
P^{(a)}(z)N^{-1}(z)=I+ O\left ( n^{-  1/5}\right )~~\mbox{as}~n \to \infty.
 \end{equation}
\end{description}

First we  make  a transformation to convert all the jumps of the RH problem for $P^{(a)}(z)$ to
constant jumps. Let us define
\begin{equation}\label{p hat a}
P^{(a)}(z)=\hat{P}^{(a)}(z)e^{n\phi_{t}(z)\sigma_3}, ~~z\in U(a_{cr},r) \backslash  \Sigma_{S},
\end{equation}
then it is readily verified  that $\hat P^{(a)}(z)$ satisfies a RH problem as follows:

\begin{description}
  \item(a)~~ $\hat{P}^{(a)}(z)$ is analytic in $U(a_{cr},r) \backslash  \Sigma_{S}$;
  \item(b)~~ In  $U(a_{cr},r)$, $\hat{P}^{(a)}(z)$ satisfies the  jump conditions
   \begin{equation}\label{P-a-hat-jump}
 \hat{P}^{(a)}_+(z)=\hat{P}^{(a)}_-(z)
 \left\{\begin{array}{ll}
          \left(
                               \begin{array}{cc}
                                 1 & \alpha  \\
                                0 & 1 \\
                                 \end{array}
                             \right), & z \in \Gamma_2, \\[.4cm]
           \left(
                               \begin{array}{cc}
                                 1 & 1-  \alpha \\
                             0& 1 \\
                                 \end{array}
                             \right), & z \in \Gamma_3, \\[.4cm]
          \left(
                               \begin{array}{cc}
                                0 &1\\
                               -1&0 \\
                                 \end{array}
                             \right),  &  z\in (a_{cr},a_{cr}+r), \\[.4cm]
           \left(
                               \begin{array}{cc}
                                 1 &0 \\
                                1 &1 \\
                                 \end{array}
                             \right), & z \in \Gamma_4 \cup \Gamma_5.
        \end{array}\right . \end{equation}

 \end{description}

We are now in a position to construct a solution for the above RH problem by using the $\Psi$-function associated with the Painlev\'{e} I equation, as introduced in Section \ref{sec-model Riemann-Hilbert problem}. We   note that $\hat{P}^{(a)}(-z)\, e^{\frac {\pi i\sigma_3}4}$ shares the same jumps with the $\Psi$-function. Then, we establish  the following conformal mapping:
\begin{equation}\label{f-function}
f(z)=\left(\frac 54\phi_{cr}(z)\right)^{2/5} \quad \textrm{for } z\in U(a_{cr},r),
\end{equation}and $r$ being sufficiently small; cf. \eqref{phi-critical function at a}. Indeed, in view of \eqref{phi function-critical},
 one can improve \eqref{phi-critical function at a} to obtain
\begin{equation}
  f(z)=-(b_{cr}-a_{cr})^{\frac 15}(2a_{cr})^{-\frac 45}(z-a_{cr})\left( 1+\frac {17}{25}\left (\frac 1{2(a_{cr}-b_{cr})}-\frac 2{a_{cr}}\right )(z-a_{cr})+ \cdots \right)
\end{equation}
as $z \to a_{cr}$. Moreover, we have
\begin{equation}\label{f expand}
f(z)^{5/2}=-\frac{5}{4}\phi_{cr}(z),
\end{equation}
where the fractional power takes the principle branch. We also define
\begin{equation}\label{q}
q(z)=-(t-t_{cr})\frac {\phi_0(z)}{f(z)^{1/2}},
\end{equation}
where the square root takes the principal branch again. It follows from the definitions of $\phi_0(z)$ and $f(z)$ in \eqref{phi function -zero} and \eqref{f-function} that  $q(z)$ is analytic in a neighborhood of $z=a_{cr}$ and
\begin{equation}\label{q-expand}
q(a_{cr})=-(2a_{cr})^{-\frac 35}(a_{cr}b_{cr})^{-\frac 12}(b_{cr}-a_{cr})^{\frac 25}(t-t_{cr}).
\end{equation}
Moreover, we have
\begin{equation}
\theta(n^{\frac 25}f(z),n^{\frac 45}q(z))=-n\phi_t(z),
\end{equation}
where $\theta(\zeta,s)$ is the function defined in \eqref{theta-def}.

With all these preparations,  the parametrix near the endpoint $a_{cr}$ can be constructed explicitly as
\begin{equation}\label{p-a}
P^{(a)}(z)=E(z)\Psi\left (n^{\frac 25}f(z),n^{\frac 45}q(z)\right )e^{-\frac {\pi i\sigma_3}4} e^{n\phi_{t}(z)\sigma_3},
\end{equation}
where $E(z)$ is defined as
\begin{equation}\label{E}
E(z)=N(z)e^{\frac {\pi i\sigma_3}4}\frac {\sigma_3+\sigma_1}{\sqrt{2}}(n^{\frac 25}f(z))^{-\frac 1 4\sigma_3 }.
\end{equation}
It is ready to see that $E(z)$ is analytic in $U(a_{cr},r)$ and the function $P^{(a)}(z)$ in \eqref{p-a} indeed satisfies the matching condition \eqref{matchingPa-N}.

{  \begin{rmk} \label{rmk-matching}
  As we have discussed in Section \ref{sec-model Riemann-Hilbert problem}, the solution $\Psi(\zeta; s)$ exists if and only if $s$ is not a pole of $y_\alpha(s)$. Then, to make $P^{(a)}(z)$ in \eqref{p-a} well-defined, we  choose $t - t_{cr} = O(n^{-4/5})$ as $n \to \infty$, and require $n^{\frac 45}q(z)$ is not a pole of $y_\alpha(s)$. Moreover, from  the large-$\zeta$ behavior of $\Psi(\zeta; s)$ in \eqref{Psi-infty}, we can verify the desired matching condition  on  $\partial U(a_{cr},r)$ in \eqref{matchingPa-N}.
\end{rmk} }

\subsection{The final transformation $S\to R$}

With all the parametrices constructed, let us introduce the final transformation
\begin{equation}\label{transformationS-R}
R(z)=\left\{ \begin{array}{ll}
                S(z)N^{-1}(z), & z\in \mathbb{C}\backslash \left \{ U(a,r)\cup U(b,r)\cup \Sigma_S \right \};\\[.1cm]
               S(z) (P^{(a)})^{-1}(z), & z\in   U(a,r)\backslash \Sigma_{S} ;  \\[.1cm]
               S(z)  (P^{(b)})^{-1}(z), & z\in   U(b,r)\backslash
               \Sigma_{S} .
             \end{array}\right .
\end{equation}
Then $R(z)$ satisfies a RH problem as follows:
\begin{description}

  \item(R1)~~ $R(z)$ is analytic in $\mathbb{C}  \backslash \Sigma_R$; see Figure \ref{contour-R};

  \item(R2)~~ $R(z)$ satisfies the  jump conditions
  \begin{equation}\label{R-jump}
    R_+(z)=R_-(z)J_{R}(z), ~~z\in\Sigma_R,
  \end{equation}
  where
    \begin{equation*}
    J_R(z)=\left\{ \begin{array}{ll}
                    P^{(a)}(z)N^{-1}(z), ~ &   z\in\partial U(a_{cr},r),\\[.1cm]
                     P^{(b)}(z)N^{-1}(z),&  z\in\partial
                    U(b,r),\\[.1cm]
    N(z)J_S(z)N^{-1}(z), ~& \Sigma_R\setminus \partial
                   ( U(a_{cr},r)\cup  U(b,r));
                 \end{array}\right .
    \end{equation*}

  \item(R3)~~  $R(z)$ satisfies the following behavior at infinity:
  \begin{equation}\label{R-infty}
  R(z)= I+ O\left({1}/{z} \right),~~\mbox{as}~ z\rightarrow\infty .
  \end{equation}

  \begin{figure}[h]
 \begin{center}
   \includegraphics[width=10cm]{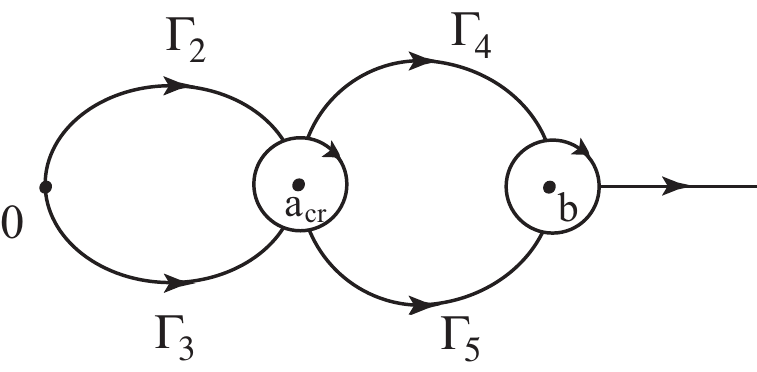} \end{center}
  \caption{Contour $\Sigma_R$}
 \label{contour-R}
\end{figure}
\end{description}

Based on the matching conditions \eqref{matchingPb-N}, \eqref{matchingPa-N} and the properties of the function $\phi_{t}(z)$ stated  in Proposition \ref{prop-phi-sign}, we have the following estimates:
\begin{equation}\label{R-jump-approx}J_R(z)=\left\{ \begin{array}{ll}
                     I+O\left (n^{-1/5}\right ),&  z\in\partial
                    U(a_{cr},r),\\[.1cm]
                     I+O\left (\frac 1n\right ),&  z\in\partial
                    U(b,r),\\[.1cm]
I+O(e^{-cn}), ~& z\in  \Sigma_R\setminus (\partial
                    U(a,r)\cup  \partial U(b,r)),
                 \end{array}\right .
\end{equation}
where $c$ is a positive constant, and the error term is uniform for $z$ on the corresponding contours.
       {  Here we also require that $t-t_{cr} = O(n^{-4/5})$; see Remark \ref{rmk-matching}.}
Then, applying the standard argument of integral operator and using
the technique of deformation of contours, see for example \cite{deift},
{  we can see that $R(z)$ exits when $n$ is large enough and $t$ is close to $t_{cr}$. Moreover, we have}
\begin{equation}\label{R-approx}
R(z)=I+O(n^{-1/5}),
\end{equation}
uniformly for $z$ in the whole complex plane.

This completes the nonlinear steepest descent analysis.


\section{Proof of the main results}\label{Proof of main results}\label{sec-proofs}

To prove our main results, Theorems \ref{Theorem: Asymptotic of Hankel} and \ref{Theorem: Asymptotic of recurrence coff}, we need more refined asymptotic approximations  for $R(z)$ than the one we get in \eqref{R-approx}.

\subsection{Asymptotic approximation  of $R(z)$}

For this purpose, we derive an asymptotic approximation for $R(z)$. To simplify the statement of our result, we denote
\begin{equation*}
  \sigma_+= \left(
              \begin{array}{cc}
                0 &1 \\
                0 & 0 \\
              \end{array}
            \right)~~\mbox{and}~~\sigma_-= \left(
              \begin{array}{cc}
                0 &0 \\
                1 & 0 \\
              \end{array}
            \right).
\end{equation*}

\begin{prop}
  Let $M=\frac 1{\sqrt{2}}(I+i\sigma_1)$ and $\mathcal{H}=\mathcal{H}(n^{4/5}q(z))$; see \eqref{P1-Hamilton}.
  {
  For $n\to \infty$, $t \to t_{cr} $ as in \eqref{s-t-double scaling}, and $s^*$ is not a pole of the tronqu\'{e}e solution $y_\alpha(s)$,}   we have
    \begin{equation}\label{R-expand}
    R(z)=I+M^{-1} \biggl( \frac {R^{(1)}(z)}{n^{1/5}} + \frac {R^{(2)}(z)}{n^{2/5}} + O\left (\frac 1 {n^{3/5}} \right ) \biggr) M,
  \end{equation}
where
\begin{equation}\label{R-1}
R^{(1)}(z)= \begin{cases}
  \D\frac {r_1\mathcal{H}(s^*)}{z-a_{cr}}\sigma_-, & z\in \mathbb{C}\setminus U(a_{cr},r), \vspace{2mm} \\
  -\D\frac {i\mathcal{H}}{\varrho^2\sqrt{f}}\sigma_+ + \left(\frac{i\mathcal{H}\varrho^2}{\sqrt{f}}+ \frac {r_1 \mathcal{H}(s^*)}{z-a_{cr}} \right)\sigma_-, & z\in U(a_{cr},r)
\end{cases}
\end{equation}
and
\begin{equation}\label{R-2}
R^{(2)}(z)=\frac {r_2}{z-a_{cr}}(y_\alpha(s^*)-\mathcal{H}^2(s^*))\sigma_3, \quad z\in \mathbb{C}\setminus  U(a_{cr},r).
\end{equation}
The constants $r_1$ and  $r_2$ in the above formulas are given as
\begin{equation}\label{r-i}
  r_1=-i\biggr( 2a_{cr}(b_{cr}-a_{cr}) \biggl)^{\frac 25},~~
  r_2=-\D \frac {a_{cr}^{4/5}}{(2(b_{cr}-a_{cr}))^{1/5}},
\end{equation}and $s^*=n^{4/5}q(a_{cr})=n^{4/5}(t_{cr}-t)(2a_{cr})^{-\frac 35}(a_{cr} b_{cr})^{-\frac 12}(b_{cr}-a_{cr})^{\frac 25}$, cf. \eqref{s-t-double scaling}.
\end{prop}

\begin{proof}

  From \eqref{Psi-infty} and \eqref{R-jump-approx}, we derive  an  asymptotic expansion for the jump $J_R(z)$
\begin{equation}\label{R-jump-expand}
J_{R}(z)=P^{(a)}(z)N^{-1}(z)=I+\sum_{k=1}^{\infty}\frac {J_{R}^{(k)}(z)}{n^{k/5}}, \qquad z\in \partial U(a_{cr},r),
\end{equation}
where
\begin{equation}\label{R-jump-expand-terms}
 J_{R}^{(k)}(z)=N(z)e^{\pi i\sigma_3/4} \, \Psi_{-k}(n^{\frac 45}q(z)) \, e^{-\pi i\sigma_3/4}N(z)^{-1}(f(z))^{-k/2},~~k=1,2,\cdots .
\end{equation}
For $z\in \partial U(b,r)$, $J_{R}(z)$ is expanded in terms of $n^{-k}$ for non-negative integers $k$. Thus, we have \eqref{R-expand} for large $n$.

To derive the explicit formulas of $R^{(1)}(z)$ and $R^{(2)}(z)$, we combine \eqref{R-jump}, \eqref{R-expand} with  \eqref{R-jump-expand}. Then  one can see that  $R^{(1)}(z)$ and $R^{(2)}(z)$ satisfy  RH problems as follows:

\begin{itemize}
  \item[(i)] $R^{(1)}(z)$ and $R^{(2)}(z)$ are analytic for $z \in \mathbb{C} \setminus \partial U(a_{cr}, r)$;

  \item[(ii)] $R^{(1)}(z)$ and $R^{(2)}(z)$ satisfy the  jump conditions
  \begin{equation}\label{R-1-RHP}
   R^{(1)}_+(z)-R^{(1)}_-(z)=\frac {i\mathcal{H}}{\varrho^2\sqrt{f}}\sigma_+-\frac{i\mathcal{H}\varrho^2}{\sqrt{f}}\sigma_-, \  z\in \partial U(a_{cr},r)
  \end{equation}
  and
  \begin{equation}\label{R-2-RHP}
   R^{(2)}_+(z)-R^{(2)}_-(z)=\frac {\mathcal{H}^2I+y_\alpha\sigma_3}{2f}+R_-^{(1)}(z)\left (\frac {i\mathcal{H}}{\varrho^2\sqrt{f}}\sigma_+ -\frac{i\mathcal{H}\varrho^2}{\sqrt{f}}\sigma_-\right ), \  z\in \partial U(a_{cr},r),
  \end{equation}

  \item[(iii)] As $z \to \infty$, both $R^{(1)}(z)$ and $R^{(2)}(z)$ are of order $O(z^{-1})$.

\end{itemize}

 It follows from  the Plemelj formula that  $R^{(1)}(z)$ and $R^{(2)}(z)$  can be represented as the Cauchy-type integrals on the circular contour $\partial U(a_{cr},r)$ of the right-hand terms in \eqref{R-1-RHP} and \eqref{R-2-RHP}, respectively.
The above RH problems can then be  solved by conducting  residue calculations of the Cauchy-type integrals. As a direct consequence we obtain \eqref{R-1} and \eqref{R-2}.
\end{proof}

\subsection{Proof of the main theorems} \label{sec-proof of theorem}

Now we are ready to derive the large-$n$ asymptotics for the recurrence coefficients and the Hankel determinant, as stated in  our main theorems, Theorems \ref{Theorem: Asymptotic of Hankel} and \ref{Theorem: Asymptotic of recurrence coff}.

\bigskip

\noindent\emph{Proof of Theorems \ref{Theorem: Asymptotic of Hankel} and \ref{Theorem: Asymptotic of recurrence coff}.} Tracing back the transformations $R \to S \to T \to Y$ in \eqref{TrsnaformationY-T}, \eqref{transformationT-S} and \eqref{transformationS-R}, we have
\begin{equation}\label{Y-trace-back}
    Y(z)=e^{\frac {nl}2\sigma_3} R(z) N(z)   e^{n(g(z)-\frac 1 2l)\sigma_3}
\end{equation}
for $z$ close to the origin. To apply the differential identities \eqref{differential identity} in Lemma \ref{lem:differential identity},
we need to extract the asymptotics of $Y(0)$ when $k=n$, $n\to \infty$,
         {  $t \to t_{cr} $ as in \eqref{s-t-double scaling}, and $s^*$ is not a pole of the tronqu\'{e}e solution $y_\alpha(s)$.}  From the explicit formula of $N(z)$ in \eqref{N-solution} and the approximation  of $R(z)$ in \eqref{R-expand}, we have
\begin{align}\label{R-0 N-0 estimate}
  R(0) N(0)=& \frac 12 \left(
               \begin{array}{cc}
                 d+\frac 1d &- i(d-\frac 1d) \\
                i( d-\frac 1d) & d+\frac 1d \\
               \end{array}
             \right)
  -\frac {r_1 \mathcal{H}(s^*) }{2a_{cr} d \, n^{1/5}}\left(
                                  \begin{array}{cc}
                                    -i & 1 \\
                                    1 & i \\
                                  \end{array}
                                \right)\nonumber\\
  &   -\frac {r_2\left (y_\alpha(s^*)-\mathcal{H}^2(s^*)\right )}{2a_{cr}n^{2/5}} \left(
               \begin{array}{cc}
                 \frac 1d-d & i(d+\frac 1d) \\
                -i( d+\frac 1d) & \frac 1d-d \\
               \end{array}
             \right)
             +O\left (\frac 1{n^{3/5}}\right ),
\end{align}
where 
$d=\left(\frac {b} {a_{cr}} \right)^{ 1/4}$ and the constants $r_j$ are defined in \eqref{r-i}. Now  \eqref{differential identity} in Lemma \ref{lem:differential identity} implies that
\begin{equation}\label{Asymptotic: H}
   \frac d{dt}H_{n,n}(t)=-n^2(R(0)N(0))_{12}(R(0)N(0))_{21}.
\end{equation}
Then the asymptotic formula  \eqref{introduction: asymptotic of  H}   for the Hankel determinant follows  immediately from the   identities \eqref{R-0 N-0 estimate} and \eqref{Asymptotic: H}.
This completes the proof of Theorem \ref{Theorem: Asymptotic of Hankel}.

To derive the asymptotics of the recurrence coefficient $\beta_{n,n}$ and the leading coefficient $\gamma_{n,n}$, we use the relations
\begin{equation}\label{recurrence coeff and Y at infinity}
   \beta_{n,n}=(Y_{-1})_{12}(Y_{-1})_{21},\qquad \gamma_{n,n}^2=-\frac {1}{2\pi i(Y_{-1})_{12}},
\end{equation}
where $Y_{-1}$ is the coefficient of the $O(1/z)$ term in the asymptotic expansion of $Y(z)z^{-n\sigma_3}$, that is,
\begin{equation}
  Y(z)z^{-n\sigma_3}=I+\sum_{j=1}^{\infty}\frac {Y_{-j}}{z^j} \quad \textrm{as } \ z \to \infty.
\end{equation}
Therefore, we also need to expand $N(z)$ and $R(z)$ as $z \to \infty.$ Again, using the explicit formula of $N(z)$ in \eqref{N-solution} and the asymptotic approximation of $R(z)$ in \eqref{R-expand}, we have
\begin{equation}\label{N at infinity}
   N(z)= I+\frac {N_{-1}}z+O\left (\frac 1{z^2}\right ), \quad \textrm{with } N_{-1}=\frac 14(b-a_{cr})M^{-1}\sigma_3M,
\end{equation}
and
\begin{equation}\label{R at infinity}
   R(z)= I+\frac {R_{-1}}z+O\left (\frac 1{z^2}\right ),
\end{equation}
where
\begin{equation}\label{R-1 estimate}
   R_{-1}=\frac {r_1 \mathcal{H}(s^*)M^{-1}\sigma_-M}{n^{1/5}}+\frac {r_2(y_\alpha(s^*)-\mathcal{H}(s^*)^2)M^{-1}\sigma_3M}{n^{2/5}}+O\left (\frac 1 {n^{3/5}}\right ).
\end{equation}
Recalling the values of $r_1$ and $r_2$ in \eqref{r-i}, and the identities $M^{-1}\sigma_-M=\frac 12 \left(
                             \begin{array}{cc}
                               -i & 1 \\
                               1 & i\\
                             \end{array}
                           \right)
$ and $M^{-1}\sigma_3M= \left(
                             \begin{array}{cc}
                             0 & i\\
                               -i & 0\\
                             \end{array}
                           \right)
$, we have from \eqref{Y-trace-back} and \eqref{recurrence coeff and Y at infinity} that
\begin{align*}
  \beta_{n,n}  & = (R_{-1}+N_{-1})_{12}(R_{-1}+N_{-1})_{21}\\
 &  =\frac {(b_{cr}-a_{cr})^2}{16}-\frac {(2a_{cr}(b_{cr}-a_{cr}))^{4/5}y_\alpha(s^*)}{4}\frac {1}{n^{2/5}}+O\left (\frac 1 {n^{3/5}}\right ).
\end{align*}This is \eqref{introduction:beta asymtotics}.

Similarly,  the asymptotics for the leading coefficient can be obtained. Indeed, we have
 \begin{align}\label{leading coefficient asymtotics}
   \gamma_{n,n}^2 =&-\frac { e^{-nl}  }{2\pi i}\frac 1{(R_{-1}+N_{-1})_{12}}\nonumber\\
=  &-\frac { e^{-nl}  }{2\pi i}\left(\frac {b_{cr}-a_{cr}}{4}i+\frac {r_1\mathcal{H}(s^*)}{2n^{1/5}}+\frac {ir_2(y_\alpha(s^*)-\mathcal{H}(s^*)^2)}{n^{2/5}}+O\left (\frac 1 {n^{3/5}}\right )\right)^{-1}\nonumber\\
=  &\frac {2 e^{-nl}}{\pi (b_{cr}-a_{cr})}\left(1-\frac {2ir_1\mathcal{H}(s^*)}{(b_{cr}-a_{cr})n^{1/5}}+\frac {4r_2(y_\alpha(s^*)-\mathcal{H}(s^*)^2)}{(b_{cr}-a_{cr})n^{2/5}}+O\left (\frac 1 {n^{3/5}}\right )\right)^{-1}\nonumber\\
=  &\frac {2 e^{-nl}}{\pi (b_{cr}-a_{cr})}\biggl(1+\frac {2ir_1\mathcal{H}(s^*)}{(b_{cr}-a_{cr})n^{1/5}} \nonumber \\
   &   -\frac {4r_1^2\mathcal{H}(s^*)^2+4r_2(y_\alpha(s^*)-\mathcal{H}(s^*)^2)(b_{cr}-a_{cr})}{(b_{cr}-a_{cr})^2n^{2/5}}+O\left (\frac 1 {n^{3/5}}\right ) \biggr),
\end{align}
which gives us \eqref{leading coefficient asymtotics-one term}.

Finally, we derive the asymptotics for $a_{n,n}$. Note that, from \eqref{a-n in terms of  Y}, we have
\begin{align}\label{a-n asymtotics}
   a_{n,n}=2\pi it\gamma_{n,n}^2Y_{11}(0)Y_{12}(0).
\end{align}
To find the  asymptotic behavior  of $Y_{11}(0)Y_{12}(0)$, we have  from \eqref{Y-trace-back} and \eqref{R-0 N-0 estimate} that
\begin{align*}
   &Y_{11}(0)Y_{12}(0)\nonumber\\
    &=e^{nl}(R(0)N(0))_{11}(R(0)N(0))_{12}\\
   &=e^{nl} \left   \{ -\frac i4\left (d^2-d^{-2}\right )-\frac {r_1 \mathcal{H}(s^*)}{2\sqrt{a_{cr} b_{cr}}n^{1/5}} \right .   \\
   &\quad \left .  -i\left (\frac {r_1^2 \mathcal{H}(s^*)^2}{4a_{cr}\sqrt{a_{cr} b_{cr}}}+\frac {r_2(y_\alpha(s^*)-\mathcal{H}(s^*)^2)(b_{cr} +a_{cr})}{2a_{cr}\sqrt{a_{cr} b_{cr}}}\right )\frac 1{n^{2/5}}+O\left (\frac 1{n^{3/5}}\right )  \right \}\nonumber\\
   &=-e^{nl}\frac {(b_{cr}-a_{cr})i}{4\sqrt{a_{cr}b_{cr}}} \left\{ 1-\frac {2ir_1 \mathcal{H}(s^*)}{(b_{cr} - a_{cr})n^{1/5}}  \right .\\
   & \quad \left .+\left ( \frac{r_1^2 \mathcal{H}(s^*)^2}{a_{cr}(b_{cr}-a_{cr})}+\frac {2r_2(y_\alpha(s^*)-\mathcal{H}(s^*)^2)(b_{cr} + a_{cr})}{a_{cr}(b_{cr}-a_{cr})}\right ) \frac{1}{n^{2/5}}+O\left (\frac 1{n^{3/5}}\right ) \right\}. \nonumber
\end{align*}
Then, the asymptotic approximation  for $a_{n,n}$ in \eqref{introduction:a-n asymtotics} follows from the above formula and \eqref{leading coefficient asymtotics}.

This completes the proof of Theorem  \ref{Theorem: Asymptotic of recurrence coff}. \hfill\qed


\section*{Acknowledgements}
{  The authors are grateful to the referee for  the valuable comments and suggestions  to improve the rigorousness  of the paper.} The work of Shuai-Xia Xu  was supported in part by the National
Natural Science Foundation of China under grant numbers
11201493  and 11571376, GuangDong Natural Science Foundation under grant numbers S2012040007824  and 2014A030313176,  and the Fundamental Research Funds for the Central Universities under grant number 13lgpy41.
Dan Dai was partially supported by   grants from the Research Grants Council of the Hong Kong Special Administrative Region, China (Project No. CityU 11300115, 11300814).
 Yu-Qiu Zhao  was supported in part by the National
Natural Science Foundation of China under grant numbers
10871212 and  11571375.


\end{document}